\documentclass[11pt]{amsart}
\usepackage{times}
\usepackage[T1]{fontenc}
\usepackage{amssymb, amsthm, amsmath}
\usepackage{mathrsfs}

\usepackage[shortlabels]{enumitem}

\usepackage[active]{srcltx}

\usepackage{todonotes}

\usepackage{setspace}
\usepackage{geometry,soul}

\usepackage{hyperref}

\DeclareMathOperator{\acl}{acl}
\DeclareMathOperator{\dcl}{dcl} 
 
 \DeclareMathOperator{\id}{Id}

 \DeclareMathOperator{\dom}{dom}

\DeclareMathOperator{\cl}{cl}

\newtheorem{introtheorem}{Theorem}

\newtheorem{theorem}{Theorem}[section]

\newtheorem{corollary}[theorem]{Corollary}

\newtheorem{fact}[theorem]{Fact}
\newtheorem{lemma}[theorem]{Lemma}

\newtheorem{proposition}[theorem]{Proposition}
\newtheorem*{gen-dif}{\fbox{{\large A}} \hypertarget{Agen-dif}{Gen-Dif}}

\newtheorem*{min-balln}{\fbox{{\large A}} \hypertarget{Amin-ball}{Cballs}}

\theoremstyle{definition}
\newtheorem{definition}[theorem]{Definition}
\newtheorem{example}[theorem]{Example}
\newtheorem{remark}[theorem]{Remark}
\newtheorem{question}[theorem]{Question}
\newtheorem{notation}[theorem]{Notation}

\newcommand{\Rr}{{\mathbb{R}}}
\newcommand{\Nn}{{\mathbb{N}}}

\newcommand{\CL}{{\mathcal L}}

\newcommand{\CM}{{\mathcal M}}

\newcommand{\CC}{{\mathcal C}}

\newcommand{\CO}{{\mathcal O}}

\newcommand{\0}{\emptyset}

\renewcommand{\phi}{\varphi}

\newcommand{\Ad}{\mathrm{Ad}}
\newcommand{\ad}{\mathrm{ad}}

\def\sub{\subseteq}

\title{On groups and fields definable in $1$-h-minimal fields}
\author{Juan Pablo Acosta L\'opez}
\address{Department of Mathematics, Ben Gurion University of the Negev, Be'er-Sheva 84105, Israel}
\email{jupaaclo1393@gmail.com }

\author{Assaf Hasson }

\address{Department of Mathematics, Ben Gurion University of the Negev, Be'er-Sheva 84105, Israel}
\email{hassonas@math.bgu.ac.il}

\date{\today}

\begin{document}
	
	\thanks{The  authors were supported by ISF grants No. 555/21.}
	
	\begin{abstract}
		We show that an infinite group $G$ definable in a $1$-h-minimal field  admits a strictly $K$-differentiable structure with respect to which $G$ is a (weak) Lie group, and show that definable local subgroups sharing the same Lie algebra  have the same germ at the identity. We conclude that infinite fields definable in $K$ are definably isomorphic to finite extensions of $K$ and that $1$-dimensional groups definable in $K$ are finite-by-abelian-by-finite. Along the way we develop the basic theory of definable weak $K$-manifolds and definable morphisms between them.  
	\end{abstract}

	\maketitle
	\section{Introduction}
	Various Henselian valued fields are amenable to model theoretic
	study. Those include  the $p$-adic numbers (more generally, $p$-adically closed fields), and (non-trivially) valued real closed and algebraically closed fields, as well as various expansions thereof (e.g. by restricted analytic functions). Recently, a new axiomatic framework for tame valued fields (of characteristic $0$) was introduced. This framework, known as Hensel-minimality\footnote{In \cite{hensel-min} and \cite{hensel-minII} various notions of Hensel-minimality -- $n$-h-minimality -- for $n\in \Nn\cup \{\omega\}$,  were introduced. For the sake of clarity of exposition, we will only discuss $1$-h-minimality.}, was suggested in \cite{hensel-min} and \cite{hensel-minII}  as a valued field analogue of  o-minimality. 
	The notion of $1$-h-minimality is both broad and powerful. Known examples include, among others, all pure Henselian valued fields of characteristic $0$ as well as their expansions by restricted analytic functions.  Known tameness consequences of $1$-h-minimality include a well-behaved dimension theory, and strong regularity of definable functions (e.g., a generic Taylor approximation theorem for definable functions).
	
	In the present paper, we initiate a study of groups definable in $1$-h-minimal fields. Using the above mentioned tameness and regularity conditions provided by $1$-h-minimality and inspired by similar studies in the o-minimal setting (initiated in \cite{Pi5}) and in $p$-adically closed fields (\cite{PilQp}) our first theorem (Proposition \ref{definable-is-lie}, stated here in a slightly weaker form) is: 
	
	\begin{introtheorem}\label{T: main intro}
		Let $K$ be a $1$-h-minimal field, $G$ an infinite group definable in $K$. Then $G$ admits a definable weak $\CC^k$ (any $k$) manifold structure with respect to which $G$ has the structure of a strictly differentiable weak $\CC^k$-Lie group. I.e., the forgetful functor from definable strictly differentiable weak Lie groups
		to definable groups is an equivalence of categories. If algebraic closure coincides with definable closure in $K$, then a definable weak Lie group is a definable Lie group. 
	\end{introtheorem}
	
	Above by a definable weak Lie group (over $K$) we mean a Lie group whose underlying $K$-manifold structure  may not have a definable (so, in particular, finite) atlas but can be covered by (the domains of) finitely many compatible \'etale maps. We do not  know whether this is a necessary requirement for the correctness of the statement, or an artifact of the proof:   we follow Pillay's argument in the o-minimal 
	and $p$-adic contexts (\cite{Pi5},  \cite{PilQp}), but the fact that in the present setting finite covers are not generically trivial, requires that we work with weakly definable manifolds, in the  above sense.   
	To pursue this argument, we have to extend the study of definable functions beyond what was done in \cite{hensel-min} (and its sequel). Specifically, instead of working with continuously differentiable functions (as is the case in the o-minimal setting) we are working with strictly differentiable functions, and for those we prove an inverse function theorem, allowing us to deduce an implicit function theorem for definable functions as well as other standard consequences of these theorems.  
	
	We do not know whether strict differentiability follows in the $1$-h-minimal context from continuous differentiability (as is the case in real analysis),  but it can be easily inferred from a multi-variable Taylor approximation theorem for definable functions available in this context. \\

	Having established that definable groups are Lie, our next theorem establishes the natural Lie correspondence (asserting that the germ of a definable group morphism at the identity is determined by its derivative at that point). For applications it is convenient to state the result for local groups (Corollary \ref{dimension-of-kernel}): 
	
	\begin{introtheorem}
		Let $K$ be a $1$-h-minimal field, $U$ and $V$  definable strictly differentiable local Lie groups and $g,f:U\to V$ definable strictly differentiable local Lie group
		morphisms. If we denote $Z=\{x\in U: g(x)=f(x)\}$, then 
		$\dim_e Z=\dim(\ker(f'(e)-g'(e)))$.
		
	\end{introtheorem}

	We then prove two applications. First, we show -- adapting techniques from the o-minimal context -- that every infinite field definable in a $1$-h-minimal field, $K$, is definably isomorphic to a finite extension of $K$, Proposition \ref{field}. This generalizes an analogous result for real closed valued fields (\cite{BaysPet}) and $p$-adically closed fields (\cite{PilQp}). It will be interesting to know whether these results can be extended to \emph{interpretable} fields (in the spirit of \cite{HaHaPeVF} or \cite[\S 6]{HrRid} under suitable additional assumptions on the RV-sort. 
	
	Our next application is a proof that definable $1$-dimensional groups are finite-by-abelian-by-finite, Corollary \ref{one-dimensional}. This generalizes  analogous results in the o-minimal context (\cite{Pi5}), in $p$-adically closed fields (\cite{PilQp}) and combines with \cite{AcostaACVF} to give a complete classification of $1$-dimensional groups definable in ACVF$_0$. \\

	The present paper is a first step toward the study of groups definable in $1$-h-minimal fields. It seems that more standard results on Lie groups over complete local fields can be extended to this context. Thus, for example, it can be shown that any definable local group contains a definable open subgroups. As the proof is long and involves new techniques we postpone it to a subsequent paper. 
	\subsection{Structure of the paper}
	In Section \ref{preliminaries} we review the basics of $1$-h-minimality and dimension theory in geometric structures. In Section \ref{taylor-section} we prove a multi-variable Taylor approximation theorem for $1$-h-minimal fields, and formulate some strong regularity conditions (implied, generically, by Taylor's theorem) that will be needed in later parts of the paper.  These results are, probably, known to the experts, and we include them mostly for the sake of completeness and clarity of exposition (as some of them do not seem to exist in writing). 
	
	In Section \ref{smooth-map} we prove the inverse function theorem and related theorems on the local structure of immersions, submersions and constant rank functions. Though some of the proofs are similar to those of analogous statements in real analysis (and, more generally, in the o-minimal context) this is not true throughout. Specifically,  $1$-h-minimality is invoked in a crucial way 
	in the proof that a function with vanishing derivative
	is locally constant, which -- in turn -- is used in our proof of the Lie correspondence for definable groups.

	Using the results of the first sections, our study of definable groups starts in 
	Section \ref{groups-section}. We first show that definable groups can be endowed with an, essentially unique,  strictly differentiable weak Lie group structure, and that the germ of  definable group morphisms are determined by their derivative at the identity. We then
	define the (definable) Lie algebra associated with a definable Lie group, and show that 
	it satisfies the familiar properties of Lie algebras. This is done using a local
	computation, after characterizing the Lie bracket as the second order part of the commutator
	function near the identity.
	
	Section \ref{fields-section} is dedicated to the classification of fields definable in $1$-h-minimal fields, and in Section \ref{one-dimensional-section} we prove our results on definable  one dimensional groups. 
	
	\section{Preliminaries}\label{preliminaries}
	In this section we set some background definitions, notation and describe basic
	relevant results, used in later sections. Most of the terminology below is either standard or taken from \cite{hensel-min}. Throughout, $K$ will denote a non-trivially valued field. We will not distinguish, notationally, between the structure and its universe. Formally, we allow $K$ to be a multi-sorted structure (with all sorts coming from from $K^{eq}$), but by a definable set we mean (unless explicitly stated otherwise) a subset of $K^n$ definable with parameters. All tuples are finite, and we write (as is common in model theory) $a\in K$ for $a\in K^n$ for $n=\mathrm{length}(a)$. We apply the same convention to variables. 
	
	To stress the analogy of the current setting with the Real numbers  we use multiplicative notation for the valuation.
	Thus, the valued group is denoted $(\Gamma,\cdot)$ and the valuation $|\cdot|:K\to \Gamma_0=\Gamma\cup \{0\}$, and if $x\in K^n$ we set $|x|:=\max_{1\leq k\leq n}|x_k|$.
	
	An open ball  of (valuative) radius $r\in \Gamma$ in $K^n$ is a set of the form
	$B=\{x\in K^n: |x-a|<r\}$ for $a\in K^n$. The balls endow $K$ with a field topology (the valuation topology). 
	Up until Section \ref{groups-section} all topological notions mentioned in the text will refer solely to this topology. 
	
	We denote $\mathcal{O}:=\{x: |x|\le 1\}$, the valuation ring,
	$\mathcal{M}:=\{x\in \CO: |x|<1\}$, the valuation ideal, and $k:=\CO/\CM$, the residue field. 
	We also denote $RV=K^{\times}/1+\mathcal{M}$. 
	More generally, whenever $s\in \Gamma$ and $s\leq 1$,  we denote
	$\mathcal{M}_s=\{x\in K: |x|<s\}$, and $RV_s=K^{\times}/1+\mathcal{M}_s$.
	If $K$ has mixed characteristic $(0,p)$, we denote 
	$RV_{p,n}=RV_{|p|^n}$ and $RV_{p,\bullet}=\bigcup_n RV_{p,n}$. \\
	
	It is convenient, when discussing approximation theorems, to adopt the big-O notation from real analysis. For the sake of clarity we recall this notation in the valued field setting:    
	
	\begin{definition}\label{big-o}
		\begin{enumerate}
			\item 
			If $f:U\to K^m$ and $g:U\to \Gamma_0$ are functions
			defined in an open neighborhood of $0$ in $K^n$, then $f(x)=O(g(x))$ means
			that there are $r, M>0$  in $\Gamma$, such that if 
			$|x|<r$ then $|f(x)|\le Mg(x)$. 
			We also denote 
			$f_1(x)=f_2(x)+O(g(x))$ if $f_1(x)-f_2(x)=O(g(x))$.
			\item
			If $g:U\to K^r$, and $s\in \mathbb{N}$, then $O(g(x)^s)=O(|g(x)|^s)$.
			\item
			If $f:Y\times U\to K^m$, is a function where $U$ is an open neighborhood
			of $0$ in $K^n$,
			and if $g:U\to \Gamma_0$, then 
			$f(y,x)=O_y(g(x))$ means that 
			for every $y\in Y$, there are $r_y,M_y>0$, 
			such that
			if $|x|<r_y$ then $|f(y,x)|\le M_yg(x)$.
		\end{enumerate}
	\end{definition}
	
	As mentioned in the introduction, in the present paper we are working with the notion of strict differentiability, which we now recall: 
	\begin{definition}
		Let $U\subset K^n$ be an open subset and $f:U\to K^m$ be a map.
		Then $f$ is  strictly differentiable at $a\in U$ if there is a linear map 
		$A:K^n\to K^m$ such that for 
		every $\epsilon>0$,
		there exists $\delta>0$ satisfying 
		$|f(x)-f(y)-A(x-y)|\leq \epsilon |x-y|$
		for every $x,y$ such that $|x-a|<\delta$ and $|y-a|<\delta$.
		
		$f$ is  strictly differentiable in $U$ if it is strictly differentiable
		at every point of $U$.
	\end{definition}
	In the situation of the definition the linear map $A$ is uniquely determined and denoted $f'(a)$.
	If $f$ is strictly differentiable in an open $U$,
	then it is continuously differentiable.
	\begin{definition}
		Let $U\subset K^n$ and $V\subset K^n$ be open subsets.
		Then $f:U\to V$ is a strict diffeomorphism if it is strictly differentiable,
		bijective and its inverse is strictly differentiable.
	\end{definition}
	
	As we will see, a strict diffeomorphism is just a strictly differentiable diffeomorphism. \\
	
	Given an open ball $B\sub K^n$ of radius $r$, 
	a subset $Y$ of $K^n$, and an element $s\in \Gamma$ with $s\leq 1$,
	we say that 
	$B$ is $s$-next to $Y$ if $B'\cap Y=\emptyset$ for $B'$ the  
	open ball of radius $s^{-1}r$ containing $B$.
	
	Note that every point not in the closure of $Y$ is contained in a ball 
	$s$-next to $Y$. This is because if $B$ is an open ball of radius $r$
	disjoint from $Y$, then every open ball of radius $sr$ contained in $B$
	is $s$-next to $Y$.
	
	Following \cite{hensel-min} we say that a finite set $Y\subset K$ prepares the 
	set $X\subset K$, if for every ball, $B$, disjoint from $Y$ is either 
	disjoint from $X$ or contained in $X$. More generally, if $s\in \Gamma$ is such that $s\leq 1$, then $Y$ $s$-prepares
	$X$ if every open ball $B$ $s$-next to $Y$ is either contained in $X$ or disjoint
	from $X$.
	
	If $K$ is a valued field of mixed characteristic $(0,p)$, given an integer
	$m\in \mathbb{N}$, an open ball, $B\sub K^n$, and a set $Y\sub K^n$,
	we say that $B$ is $m$-next to $Y$ if it is $|p|^m$-next to $Y$. Similarly,  if $s\in \Gamma$ and $s\leq 1$ then $B$ is 
	$m$-$s$-next to $Y$ if it is $|p|^ms$-next to $Y$. Given a finite $Y\subset K$ and $X\subset K$, 
	we say that $Y$ $m$-prepares (resp. $m$-$s$-prepares) the set
	$X$ if $Y$ $|p|^m$-prepares $X$ (resp. $Y$ $|p|^ms$-prepares $X$).
	
	Next, we recall the definitions of 1-h-minimality defined in the equi-characteristic $0$ (\cite{hensel-min}) and in mixed characteristic (\cite{hensel-minII}) settings: 
	
	\begin{definition}
		Let $K$ be an $\aleph_0$-saturated non-trivially valued field of characteristic $0$, which is a structure
		in a language extending the language of valued fields. 
		\begin{enumerate}
			\item If $K$ has residue characteristic $0$ then $K$ is  $1$-h-minimal, if for any $s\le 1$ in  $\Gamma$ any $A\sub K$, $A'\in RV_s$ (a singleton) and every $(A\cup RV\cup A')$-definable 
			set $X\subset K$, there is an $A$-definable finite set $Y\subset K$  $s$-preparing $X$. 
			\item If $K$ has mixed characteristic $(0,p)$ then $K$ is $1$-h-minimal, if for any $s\le 1$ in  $\Gamma$ any $A\sub K$, $A'\in RV_{s}$ (a singleton) and every $(A\cup RV_{p,\bullet}\cup A')$-definable 
			set $X\subset K$, there is $m\in \Nn$ and an $A$-definable finite set $Y\subset K$ which $m$-$s$-prepares $X$. 
		\end{enumerate}
	\end{definition}
	In the sequel, when appealing directly to the definition,  we will only need the case $s=1$ (so $A'$ does not appear).
	The parameter $s$ does appear implicitly, though, when applying properties of $1$-h-minimality such as  generic continuity of definable functions (see \cite[Proposition 5.1.1]{hensel-min}).
	
	Below we will need to study properties of  ``one-to-finite definable functions'' (definable correspondences, in the terminology of \cite{SimWal}). 
	It turns out that statements regarding such objects can sometimes 
	be reduced to statements on definable functions
	in  expansions of the language by algebraic Skolem functions (i.e., Skolem functions for definable finite sets). For this, the following will be convenient (see  \cite[Proposition 4.3.3]{hensel-min}, and 
	\cite[Proposition 3.2.2]{hensel-minII}): 
	
	\begin{fact}\label{acl=dcl}
		Suppose $K$ is a 1-h-minimal valued field. 
		Then there exists a language $\mathcal{L}'\supseteq \mathcal{L}$, an elementary extension $K'$ of $K$, 
		and an $\aleph_0$-saturated $\mathcal{L}'$-structure on $K'$ extending the 
		$\mathcal{L}$-structure of $K'$, such that 
		$K'$ is 1-h-minimal as an $\mathcal{L}'$-structure, and such that 
		$\acl_{\mathcal{L}'}(A)=\dcl_{\mathcal{L}'}(A)$ for all $A\sub K'$.
	\end{fact}
	
	Above and throughout, algebraic and definable closures are always assumed to be taken in the $K$ sort. In the sequel we will refer to the property appearing in the conclusion of Fact \ref{acl=dcl} simple as "$\acl=\dcl$". 
	
	\begin{remark}\label{acl=dcl-remark}
		Given an $\CL$-definable set, $S$, statements concerning topological or geometric properties  of 
		$S$ are often expressible by first-order $\CL$-formulas. As the topology on $K$ is definable in the valued field language, and the dimension of definable sets in $1$-minimal minimal structure is determined by the topology (see Proposition \ref{dimension}),  the truth values of the hypothesis and conclusion of such statements (for our fixed $\CL$-definable set $S$) 
		are the same in $K$ and in any elementary extension  $K\prec K'$, as well as in any $1$-h-minimal expansion of the latter. Therefore, by Fact \ref{acl=dcl}, in the proof of such statements (for a fixed definable $S$) there is no harm assuming  
		$\acl=\dcl$.
	\end{remark}
	\subsection{Geometric structures} Geometric structures were introduced in  
	\cite[\S 2]{HruPil}. Let us recall the definition: 
	An  $\aleph_0$-saturated structure, $M$ is \emph{pregeometric} if  $\acl(\cdot)$  is a matroid, that is, it satisfies the exchange property:
	\[ \text{if } a\in \acl(Ab)\setminus \acl(A), \text{ then } b\in \acl(Aa) \text{ for singletons } a,b\in M. \]
	
	In this situation the matroid gives  a notion of dimension, $\dim(a/b)$, 
	the dimension of a tuple $a$ over a tuple $b$
	as the smallest length of a sub-tuple $a'$ of $a$ such that $a\in \acl(a'b)$,
	and the dimension of a $b$-definable set $X$ as the maximum of the dimensions
	$\dim(a/b)$ with $a\in X$ (this does not depend on $b$).
	As is customary, we set $\dim(\emptyset)=-\infty$. We recall the basic properties of dimension (see \cite[\S 2]{HruPil} for all references). 
	This dimension satisfies the additivity property 
	\[
	\dim(ab/c)=\dim(a/bc)+\dim(b/c),
	\]
	that we will invoke without further reference. 
	We call $a$ and $b$ algebraically independent over $c$ if 
	$\dim(a/bc)=\dim(a/c)$.  Note that by additivity of dimension this is a symmetric relation. Note also that additivity implies that
	if $b,c$ are inter-algebraic over $a$, meaning $b\in \acl(ac)$ and 
	$c\in \acl(ab)$, then $\dim(b/a)=\dim(c/a)$ (in particular, this holds when 
	$c$ is the image of $b$ under an $a$-definable bijection).
	If $M$ is a pregeometric structure and $f:X\to Y$ is a surjective definable
	function with fibers of constant dimension $k$, then $\dim(X)=\dim(Y)+k$.
	This is a consequence of the additivity formula.
	
	Given an $a$-definable set $X$, a generic element of $X$
	over $a$ is an element $b\in X$ such that $\dim(b/a)=\dim(X)$. By compactness generic elements can always be found in saturated enough models. We call $Y\subset X$ large if $\dim(X\setminus Y)<\dim(X)$.
	This is equivalent to $Y$ containing every generic point of $X$.
	
	A pregeometric structure, $M$, eliminating the quantifier $\exists^\infty$ is called geometric. If $M$ is geometric dimension is definable in definable families. Namely,  for $\{X_a\}_{a\in S}$, 
	a definable family, the set $\{a\in S: \dim(X_a)=k\}$ is definable.
	\iffalse
	It is not hard to see that the property of an $\aleph_0$-structure being geometric is preserved under reducts. 
	\fi
	
	The following simple fact is a translation of the definition of a 
	pregeometry to a
	property of definable sets. Note as an aside that this reformulation implies
	that the property of being a pregeometry is preserved under reducts.
	I.e., if $M$ is an $\aleph_0$-saturated pregeometric 
	$\mathcal{L}'$-structure,
	and $\mathcal{L}\subset \mathcal{L}'$, then $M$ is also a pregeometric 
	$\mathcal{L}$-structure. 
	For the sake of completeness, we give the proof: 
	\begin{fact}\label{pregeometry-definable-set-criterion}
		Suppose $M$ is an $\aleph_0$-saturated structure.
		Then $M$ is pregeometric if and only if for every definable  $X\subset M\times M$ if the projection, $\pi_1:X\to M$,  into the first factor is finite-to-one, and  $\pi_2:X\to M$ is  the projection into the second factor, then the set
		$Y=\{c\in M: \pi_2^{-1}(c)\cap X\text{ is infinite}\}$ is finite.
	\end{fact}
	\begin{proof}
		Suppose $M$ is pregeometric and suppose $X\subset M\times M$ is $A$-definable such that
		$\pi_1^{-1}(x)\cap X$ is finite for all $x\in M$.
		Suppose also that $Y=\{y\in M: \pi_2^{-1}(y)\cap X\text{ is infinite}\}$
		is infinite. By compactness and saturation, we can choose 
		$b\in Y$ such that $\dim(b/A)=1$. Similarly, we can find $a\in \pi_2^{-1}(b)\cap X$ such that $\dim(a/Ab)=1$. We conclude that $\dim(ab/A)=2$, and so 
		$\dim(X)\ge 2$. This contradicts the fact that $\pi_1^{-1}(x)\cap X$ is finite
		for all $x\in M$.
		
		For the converse, 
		suppose $A$ is a finite subset of $M$ and $a,b\in M$ are singletons such that
		$a\in\acl(Ab)\setminus \acl(A)$.
		Then there is an $A$-definable set $X\subset M\times M$ such that
		$(b,a)\in X$ and $\pi_1^{-1}(b)\cap X$ 
		is finite, say of cardinality $k$. If we take 
		$Z=\{c\in M: \pi_1^{-1}(c)\cap X\text{ has cardinality }k\}$, then we may
		replace $X$ by
		$X\cap Z\times M$, and we may assume that $\pi_1^{-1}(c)\cap X$ 
		is either empty or of constant
		finite cardinality for all $c\in M$. 
		In this case, by the hypothesis we conclude that 
		$Y=\{y\in M: \pi_2^{-1}(y)\cap M\text{ is infinite}\}$ is finite.
		Note that $Y$ is $A$-invariant and definable, so it is $A$-definable.
		We conclude that $a\notin Y$, because $a\notin \acl(A)$, and so 
		$b\in \acl(Aa)$ as required.
	\end{proof}
	
	The next characterization of the $\acl$-dimension should be well known: 
	\begin{fact}\label{dimension-is-canonical}
		Suppose $M$ is an $\aleph_0$-saturated structure, which eliminates the $\exists^{\infty}$
		quantifier.
		Suppose there is a function, $X\mapsto d(X)$, from the non-empty definable subsets of (cartesian powers of) $M$ into 
		$\mathbb{N}$  satisfying:
		\begin{enumerate}
			\item If $X\subset M^n\times M$ is such that the first coordinate projection 
			$\pi_1:X\to M^n$ is finite to one, then $d(X)=d(\pi_1(X))$.
			\item If $X\subset M^n\times M$ is such that the first coordinate projection 
			$\pi_1:X\to M^n$ has infinite fibers, then $d(X)=d(\pi_1(X))+1$.
			\item If $\pi:M^n\to M^n$ is a coordinate permutation, then 
			$d(X)=d(\pi(X))$.
			\item $d(X\cup Y)=\max\{d(X),d(Y)\}$.
			\item $d(M)=1$
			\item $d(X)=0$ if and only if $X$ is finite.
		\end{enumerate}
		Then $M$ is a geometric structure and $d$ coincides with its $\acl$-dimension.
	\end{fact}
	\begin{proof}
		It suffices to show that $M$ is pregeometric.  We  use  
		Fact \ref{pregeometry-definable-set-criterion}.
		Let $X\subset M\times M$ be such that $\pi_1^{-1}(x)\cap X$ is finite for all
		$x\in M$. Take $Y=\{y\in M: \pi_2^{-1}(y)\cap X\text{ is infinite}\}$.
		Because $M$ eliminates the $\exists^{\infty}$ quantifier we have that 
		$Y$ is definable. If $Y$ is infinite we conclude that 
		$d(X)\ge d(X\cap \pi_2^{-1}(Y))=d(Y)+1=2$.
		The first inequality by item (4), the second equality 
		by items (3) and (2), and the third
		by item (6).
		On the other hand $d(X)=d(\pi_1(X))\le d(M)=1$, the first equality by item (1),
		the second inequality by item (4) and the third by item (5). This is a 
		contradiction and finishes the proof. 
		
		In order to see that $d(X)=\dim(X)$ for $X\subset M^n$ we may proceed by 
		induction on $n$. The base case $n=1$ follows from item (4), (5) and (6).
		So suppose that $X\subset M^n\times M$. Denote 
		$Y=\{x\in M^n: \pi_1^{-1}(x)\cap X\text{ is infinite}\}$.
		By hypothesis $Y$ is a definable set.
		Denote $X_1=\pi_1^{-1}(Y)\cap X$ and $X_2=X\setminus X_1$.
		Then by items (1),(2) and (4) we conclude that 
		$d(X)=\max\{d(X_1),d(X_2)\}=\max\{d(Y)+1,d(\pi_1(X)\setminus Y)\}$.
		For the same reason we have the formula
		$\dim(X)=\max\{\dim(Y)+1,\dim(\pi_1(X)\setminus Y)\}$, so 
		$d(X)=\dim(X)$ as required.
	\end{proof}
	
	The next fact is also standard: 
	
	\begin{fact}\label{cell-decomposition0}
		Suppose $M$ is a geometric structure. Suppose $X\subset M^n$ is $a$-definable.
		Then there is a partition of $X$ into a finite number of $a$-definable sets
		$X=X_1\cup\cdots\cup X_n$, such that for each member of the partition $X_k$, there is
		a coordinate projection $\pi:X_k\to M^r$ 
		which is finite to one and has image of dimension $r$.
	\end{fact}
	\begin{remark}\label{terminology-cell-decomposition-0}
		For this statement we need to allow the identity $\mathrm{id}:M^n\to M^n$ as a 
		coordinate projection.
		
		Also, recall that $M^{0}$ is a set consisting of one element.
		For this statement we also need to 
		allow the constant function $M^n\to M^{0}$ as a coordinate projection.
	\end{remark}
	\begin{proof}
		By induction on the dimension of the ambient space $n$.
		Consider the projection onto the first $n-1$ coordinates
		$\pi_1:M^n\to M^{n-1}$. 
		Then the set $Y\subset M^{n-1}$ of $y$ such that the fibers
		$X_y=\pi_1^{-1}(y)\cap X$ are infinite is definable.
		So partitioning $X$ we may assume all the nonempty fibers of $X$
		over $M^{n-1}$ are finite, or all are infinite.
		If all the fibers of $X\to K^{n-1}$ are finite then we finish by induction.
		
		If all the nonempty fibers are infinite then by the induction
		there is a partition $Y=\bigcup_i Y_i$ and for each $Y_i$ there is a coordinate
		projection $\tau: Y_i\to K^r$ 
		with finite fibers
		and $r=\dim(Y_i)$. 
		Denote $\pi_2:M^{n}\to M$ the projection onto the last coordinate. 
		Then setting $X_i=X\cap \pi_1^{-1}(Y_i)$,
		the projection $\pi(x)=(\tau(\pi_1(x)),\pi_2(x))$ has the desired properties.
	\end{proof}
	
	The next proposition is key. It asserts that $1$-h-minimal fields are geometric, and it connects (combined with the previous fact) topology and dimension in such structures:  
	
	\begin{proposition}\label{dimension}
		Suppose $K$ is a 1-h-minimal valued field. Then:
		\begin{enumerate}
			\item
			$K$ is a geometric structure.
			\item
			Every $X\subset K^n$ satisfies $\dim(X)=n$ if and only if
			$X$ has nonempty interior. Every $X\subset K^n$ satisfies $\dim(X)<n$ if and only if 
			$X$ is nowhere dense.
			\item
			For $X\subset K^n$,
			we have $\dim(X)=\max\limits_{x\in X}\dim_x(X)$, where we denote
			$\dim_x(X)$ the local dimension of $X$ at $x$, defined as 
			$\dim_x(X)=
			\min\{\dim(B\cap X): x\in B \text{ is an open ball}\}$
		\end{enumerate}
	\end{proposition}
	\begin{proof}
		This is essentially items (1)-(5) of 
		\cite[Proposition 5.3.4]{hensel-min} in residue characteristic
		$0$ and contained in \cite[Proposition 3.1.1]{hensel-minII} in mixed characteristic.
		
		For example, assume $K$ has residue characteristic $0$.
		That $K$ is geometric is proved
		in the course of the proof of \cite[Proposition 5.3.4]{hensel-min}.
		We can also derive it from Fact \ref{dimension-is-canonical}
		and \cite[Proposition 5.3.4]{hensel-min}. 
		
		The topological characterization $\dim(X)=n$ if and only if $X$ has nonempty
		interior, is a particular case of item (1) in \cite[Proposition 5.3.4]{hensel-min}.
		That $\dim(X)<n$ if and only if $X$ is nowhere dense follows from this.
		Indeed, if $\dim(X)=n$, then $X$ has nonempty interior and so it is not nowhere
		dense. If $\dim(X)<n$ and $U\subset K^n$ is nonempty open, then 
		$\dim(U\setminus X)=n$, and so $U\setminus X$ has nonempty interior.
		This implies $X$ is nowhere dense.
		
		That dimension is the maximum of the local dimensions is item (5) of Proposition 5.3.4 of 
		\cite{hensel-min}
	\end{proof}
	\begin{proposition}\label{generic-continuity}
		Suppose $K$ is a 1-h-minimal field.
		Suppose $f:U\to K^m$ is a definable function. Then there is a definable open dense 
		subset $U'\subset U$ such that $f:U'\to K^m$ is continuous.
	\end{proposition}
	\begin{proof}
		This is essentially a particular case of
		\cite[Proposition 5.1.1]{hensel-min} in residue characteristic $0$,
		and contained in \cite[Proposition 3.1.1]{hensel-minII} in mixed characteristic.
		
		Indeed, because the intersection of open dense sets is open and dense we reduce
		to the case $m=1$. 
		From those propositions one gets that the set $Z$ of points where $f$ is continuous
		is dense in $U$. As $Z$ is not nowhere dense we  conclude using  item (2) of Proposition \ref{dimension} that $\dim(Z)=n$ and so
		$Z$ has nonempty interior.
		If $V\subset U$ is a nonempty open definable subset then, 
		as $Z\cap V$ is the set of points at which $f|_{V}$ is continuous, by what we just proved
		$Z\cap V$ has nonempty interior. 
		We conclude that the set of points at which $f$ is continuous
		has a dense interior in $U$, as desired.
	\end{proof}
	Next we describe a topology for $Y^{[s]}$, the set of subsets of $Y$ of
	cardinality $s$, for $Y$ a Hausdorff topological space,
	and $s$ a positive integer.  We prove a 
	slightly more general statement that will be applied when  $X$ is $Y^s\setminus \Delta$, the set
	of tuples of $Y^s$ with distinct coordinates and the symmetric group, $S_s$, on $s$ elements acting on $Y^s$ by coordinate permutation, in which case the orbit space is identified with $Y^{[s]}$. 
	\begin{fact}\label{topology-group-action}
		Suppose $X$ is a Hausdorff topological space, and $G$ a finite group
		acting on $X$ by homeomorphisms, and such that every $x\in X$ has a trivial
		stabilizer in $G$.
		Then $X/G$ equipped with the quotient topology is Hausdorff and
		the map $p:X\to X/G$ is a closed finite covering map. In fact for every $x\in X$
		there is an open set $x\in U\subset X$ such that $\{gU: g\in G\}$
		are pairwise disjoint, $p^{-1}p(U)=\bigcup_ggU$ and $p|_{gU}$
		is a homeomorphism onto $p(U)$.
	\end{fact}
	\begin{proof}
		We know that $p$ is open, since $p^{-1}p(U)=\bigcup_{g\in G}gU$ is open for $U$ open.
		Consider the orbit $\{gx\}_{g\in G}$ of $x$. By assumption, if $g\neq h$, then
		$gx\neq hx$. Let $V$ be an open set in $X$ containing $\{gx\}_{g\in G}$. 
		Now, because $X$ is Hausdorff, we conclude that there are 
		$U_g$ open neighborhoods of $gx$, contained in $V$,
		such that $U_g\cap U_h=\emptyset $ for $g\neq h$.
		If we take $U=\bigcap_{g\in G}g^{-1}U_g$, then $gU\subset U_g$ and so 
		$\{gU: g\in G\}$ are pairwise disjoint. We conclude that $p$ is closed and restricted
		to $gU$ is a homeomorphism.
		That $X/G$ is Hausdorff now follows from this. Indeed, if $p(x)\neq p(y)$, then there
		are open sets $V_1$ and $V_2$ of $X$, which are disjoint and such that 
		$p^{-1}p(x)\subset V_1$
		and $p^{-1}p(y)\subset V_2$. Because $p$ is closed there are open sets $p(x)\in U_1$
		and $p(y)\in U_2$ in $X/G$ such that $p^{-1}(U_i)\subset V_i$. We conclude that $U_1$
		and $U_2$ are disjoint.
	\end{proof}
	\begin{proposition}\label{generic-regularity-finite}
		Let $K$ be a 1-h-minimal valued field. 
		Suppose $U\subset K^n$ is open and $f:U\to (K^r)^{[s]}$ is definable.
		Then there is an open dense definable set 
		$U'\subset U$ such that $f$ is continuous in $U'$.
	\end{proposition}
	\begin{proof}
		This statement is equivalent to saying that the interior of the set of points
		on which $f$ is continuous is dense. 
		As this property is expressible by a first order formula we may assume
		$\acl=\dcl$, see Fact \ref{acl=dcl} and the remark following it. 
		
		In that case we have a definable section
		$s:(K^r)^{[s]}\to K^{rs}$, and if $V\sub U$ is open dense such that 
		$sf$ is continuous, as provided by Proposition \ref{generic-continuity}, then $f$ is continuous in $V$.
	\end{proof}
	\begin{proposition}\label{cell-decomposition}
		Suppose $X\subset K^n$ is $b$-definable. Then there is finite partition 
		of $X$ into $b$-definable sets, such that for each element $Y$
		of the partition
		there is a coordinate projection $\pi:Y\to U$ onto an open set $U\subset K^m$,
		such that the fibers of $\pi$ all have the same cardinality equal to $s$,
		and the associated map $f:U\to (K^{n-m})^{[s]}$ is continuous.
	\end{proposition}
	\begin{remark}
		As in Remark \ref{terminology-cell-decomposition-0}, we need to allow 
		the two cases $m=0$ and $m=n$. The set $K^0$ consists of a single point,
		and has a unique topology.
	\end{remark}
	\begin{proof}
		This is a consequence of dimension theory and the previous observation.
		In more detail, we proceed by induction on  the dimension of $X$.
		
		First, recall that $X$ 
		has a finite partition into $b$-definable sets such 
		that for each set $X'$ in the partition there is a coordinate projection 
		$\pi:X'\to K^r$ with finite fibers and $r=\dim(X')$, see 
		Fact \ref{cell-decomposition0}.
		
		So now assume $\pi:X\to K^r$ is a coordinate projection with finite fibers
		and $r=\dim(X)$, and denote $\pi':X\to K^{n-r}$ 
		the projection into the other coordinates. 
		There is an integer $s$ which bounds the cardinality
		of the fibers of $\pi$. If we denote $Y_k$ the set of elements $a\in K^r$
		such that $X_a=\pi'(\pi^{-1}(a))$ has cardinality $k$ then we get
		$Y_0\cup\cdots\cup Y_s=K^r$.
		Now let $V_j\subset Y_j$ be open dense in the interior of $Y_j$ and such that
		the map $V_j\to (K^{n-r})^{[j]}$ given by $a\mapsto X_a$ 
		is continuous, see Proposition \ref{generic-regularity-finite}.
		Then the set $\{x\in X: \pi(x)\in Y_j\setminus V_j, 1\leq j\leq s\}$
		is of lower dimension than $X$, by item 2 of Proposition \ref{dimension},
		and so we may apply the induction hypothesis on it. 
	\end{proof}
	Recall that a subset $Y\subset X$ of a topological space $X$ is locally closed
	if it is the intersection of an open set and a closed set.
	This is equivalent to $Y$ being relatively open in its closure.
	It is also equivalent to, for every point $y\in Y$,
	the existence of a neighborhood $V$ of $y$, such that 
	$Y\cap V$ is relatively closed in $V$.
	\begin{proposition}\label{definable-is-constructible}
		Suppose $K$ is 1-h-minimal and $X\subset K^n$ an $a$-definable set. Then $X$ is a finite union of $a$-definable locally closed subsets of $K^n$.
	\end{proposition}
	\begin{proof}
		This is a consequence of Proposition \ref{cell-decomposition}.
		Namely, there is a partition of $X$ into a finite union
		of definable subsets for each of which there is a coordinate projection with finite fibers onto
		an open set $U$, so we may assume $X$ is of this form.
		We may further assume that the fibers
		have constant cardinality $k$ and the associated mapping
		$U\to (K^r)^{[k]}$ is continuous. 
		Then $X$ is closed in $U\times K^r$ and so locally closed.
	\end{proof}
	We finish by reviewing a more difficult property of dimension. We will only
	use this in Proposition \ref{generic-embedding},
	Proposition \ref{subgroups-are-closed}
	and Corollary \ref{subgroup-tangent}, which are not used in the main
	theorems.
	\begin{proposition}\label{dimension-boundary-function}
		Suppose $K$ is a 1-h-minimal field and 
		$X\subset K^n$. Then $\dim(\cl(X)\setminus X)<\dim(X)$.
	\end{proposition}
	This is item 6 of \cite[Proposition 5.3.4]{hensel-min} for the residue
	characteristic $0$ and it is contained in Proposition 3.1.1 of 
	\cite{hensel-minII} in the mixed characteristic case.

	\section{Taylor approximations}\label{taylor-section}
	
	In this section we show that, in the $1$-h-minimal setting, the generic one variable 
	Taylor approximation theorem  (\cite[Theorem 3.1.2]{hensel-minII})
	implies a multi-variable version of the theorem. 
	In equicharacteristic $(0,0))$ this is  \cite[Theorem 5.6.1]{hensel-min}. 
	Though the proof in mixed characteristic is, essentially, similar;
	we give the details, for the sake of completeness and in view of the importance of this result in the sequel.
	
	We then proceed to introducing some regularity conditions for definable functions (implied, in the present context,  by Taylor's approximation theorem) necessary for computations related to the Lie algebra of definable groups. 
	
	First, we recall the multi-index notation.
	If $i=(i_1,\dots,i_n)\in \mathbb{N}^n$, we denote
	$|i|=i_1+\cdots+i_n$ and $i!=i_1!\cdots i_n!$. 
	For $x=(x_1,\cdots,x_n)\in K^n$ we denote $x^i=x_1^{i_1}\cdots x_n^{i_n}$.
	Also if $f:U\to K$ is a function defined in an open set of $K^n$,
	we denote
	$f^{(i)}(x)=(\frac{\partial^{i_1}}{\partial x_1^{i_1}}
	\cdots \frac{\partial^{i_n}}{\partial x_n^{i_n}}f)(x)$ whenever it exists.
	Note that we are not assuming equality of mixed derivatives, but see
	Corollary \ref{partial-derivatives-commute}.
	\begin{proposition}\label{taylor0}
		Let $K$ be a 1-h-minimal field of residue characteristic $0$.
		Suppose $f:U\to K$ is an $a$-definable function with $U\subset K^n$ open and 
		let $r\in \Nn$.
		Then there is an $a$-definable set $C$, of
		dimension strictly smaller than $n$, such that for any open ball, $B\sub U$ 
		disjoint from $C$ the derivative $f^{(i)}$ exists in $B$ for every $i$ with 
		$|i|\leq r$, and has constant valuation in $B$. Moreover, 
		
		\begin{displaymath}
			\left | f(x)-\sum_{\{i:|i|<r\}}\frac{1}{i!}f^{(i)}(x_0)(x-x_0)^i \right |
			\leq \max_{\{i:|i|=r\}} \left |\frac{1}{i!}f^{(i)}(x_0)(x-x_0)^i \right|
		\end{displaymath}
		For every $x,x_0\in B$.
	\end{proposition}
	This is \cite[Theorem 5.6.1]{hensel-min}. Our first order of business is to generalize this result to positive residue characteristic. 
	
	The following fact is proved by a standard compactness argument, 
	and is often applied implicitly. We add this argument for convenience.
	%\textcolor{red}{Here is a suggested reformulation: "Let $M$ be a an $\aleph_0$-saturated structure, $\{\Phi^l(\bar D\}_{l\in I}$ a family of properties of definable sets $\bar D$ in $M$, indexed by a directed set $I$. Assume that $\Phi^l\vdash \Phi^k$ for $l<k\in I$ and that  that for all $l\in I$ the property $\phi^l$ is definable in definable families. I.e, if $\bar D_a$ is a family of definable sets depending on a parameter $a$ then the set $\{a\in M: D_a \text{ has property } \Phi^l\}$ is a definable set...". I find it a little more readable. The current notation seems overly complicated.} 
	\begin{fact}\label{standard-compactness}
        Let $M$ be an $\aleph_0$-saturated structure, and
        $\{\Phi^l(\bar{D})\}_{l\in I}$ be a family of properties of definable sets $\bar{D}=(D_1,\dots,D_n)$ in $M$, indexed 
		by a directed set $I$. Let $b$ be a tuple in $M$, and $S$ be a $b$-definable set. Assume that 
        \begin{enumerate}
            \item 
                For all $l$ the property $\Phi^l$ is  definable in definable families. I.e, 
		          if $\{D_{i,a}\}_{a\in T}$ are $b$-definable families, 
		          then the set $\{a\in T: \Phi^l({\bar{D}_a})\text{ holds}\}$ is $b$-definable.
            \item
		           $\Phi^l$ implies $\Phi^{l'}$ for all $l\leq l'$. 
		      \item
		          For every $a\in S$,
		          there are $ba$-definable sets $D_{i,a}$, satisfying $\Phi^{l_a}$ for some  $l_a\in I$. 
        \end{enumerate}
        
	    Then there are $\{D_{i,a}\}_{a\in S}$
	    $b$-definable families of sets, and a fixed $l\in I$,
	    such that $\Phi^l(\bar{D}_{a})$ holds for 
	    every $a\in S$.
	\end{fact}
	\begin{remark}
		Formally, $\Phi^l$ is a subset of 
		\[
		\{(D_1,\dots,D_n): D_i \text{ is a definable set and}\}
		\] 
		and we say 
		$\Phi^l(D_1,\dots,D_n)$ holds if the tuple
		$(D_1,\cdots,D_n)$ belongs to $\Phi^l$.
		
		Note also that
		the tuple $(D_1,\dots,D_n)$ can be replaced with $D_1\times \cdots \times D_n$,
		so there is no loss of generality in taking $\Phi^l$ of the form
		$\Phi^l(D)$. 
	\end{remark}
	\begin{proof}
		Let $a\in S$. By hypothesis there are  $b$-definable families
		$\{D_{i,a'}^a\}_{a'\in S_i^{0,a}}$
		and an element $l^a\in I$,
		such that $(D_{1,a}^a,\dots,D_{n,a}^a)$ satisfies $\Phi^{l^a}$.
		Consider $S^a$ to be the set of $a'\in S$ such that $a'\in S_i^{0,a}$ for $i=1,\dots,n$,
		and such that 
		$\Phi^{l^a}(\bar{D}_{a'}^a)$ holds.
		By hypothesis, this is a $b$-definable set contained in $S$ and containing $a$.
		
		We conclude that $S=\bigcup_{a\in S}S^a$ is a cover of $S$ by $b$-definable sets,
		and so by compactness and saturation there is a finite sub-cover, say
		$S=S^1\cup\dots \cup S^k$ for $S^r=S^{a_r}$. Indeed, if there was no finite 
		sub-cover,
		then the partial type expressing $x\in S$ and $x\notin S^a$ for all $a\in S$, is a 
		consistent $b$-type, and so a realization in $M$ would contradict $S=\bigcup_{a\in S}S^a$.
		
		Then $D_{i,a}$ defined as $D_{i,a}^{a_r}$ if $a\in S^r\setminus \bigcup_{r'<r}S^{r'}$
		satisfy that $\{D_{i,a}\}_{a\in S}$ form $b$-definable
		families. If we take $l$ such that $l\geq l^{a_1},\dots,l^{a_k}$, then we get that
		$\Phi^l(\bar{D}_{a})$ holds for every $a\in S$, as required.
	\end{proof}
	\begin{notation}
		If $D\subset E\times F$, and $a\in E$ we often denote
		$D_a=\{b\in F: (a,b)\in D\}$. If $b\in F$ we denote, when no ambiguity can occur, 
		$D_b=\{a\in E: (a,b)\in D\}$.
		If $f:D\to C$ is a function, we let  $f_a:D_a\to C$ denote the function $f_a(b)=f(a,b)$,
		and similarly $f_b$ for $b\in F$.
	\end{notation}
	\begin{proposition}\label{taylor}
		Let $K$ be a 1-h-minimal field of positive residue characteristic,  $f:U\to K$ 
		an $a$-definable function with $U\subset K^n$ open, and let  
		$r\in \Nn$.
		Then there is an integer $m$, and a set $C$,
		closed, $a$-definable with  $\dim(C)<n$, such that for every open ball,
		$B\sub U$
		$m$-next to $C$, $f^{(i)}$ exists in $B$ for every $i$ with 
		$|i|\leq r$, and $f^{(i)}$ has constant valuation in $B$. Moreover, 
		
		\begin{displaymath}
			\left | f(x)-\sum_{\{i:|i|<r\}}\frac{1}{i!}f^{(i)}(x_0)(x-x_0)^i \right |
			\leq \max_{\{i:|i|=r\}} \left |\frac{1}{i!}f^{(i)}(x_0)(x-x_0)^i \right |
		\end{displaymath}
		For every $x,x_0\in B$.
	\end{proposition}
	\begin{proof}
		We proceed by induction on $n$, the case $n=1$ being \cite[Theorem 3.1.2]{hensel-minII}. Assume the result for $n$ and let $f:U\to K$ be an 
		$a$-definable function with  $U\subset K^n\times K$ open, $i$ a multi-index with $|i|\le r$. 
		Then for every $x\in K^n$ there is a finite $ax$-definable set
		$C_x\subset K$ and an integer $m_x$ such that 
		\begin{equation}\label{taylor-1}
			|f_x(y)-\sum_{s<r}\frac{1}{s!}f_x^{(s)}(y_0)(y-y_0)^s|
			\leq |\frac{1}{r!}f_x^{(r)}(y_0)(y-y_0)^r|
		\end{equation}
		for every $y$ and $y_0$ in an open ball $m_x$-next to $C_x$, and such that
		$|f_x^{(s)}(y)|$ exists and is constant in any such open ball.
		By a standard compactness argument (See Fact \ref{standard-compactness}),  we may assume that the $C_x$
		are uniformly  definable and that there is some $m\in \Nn$ such that $m_x=m$ for all $x$. 
		Define $C=\cup_x\{x\}\times C_x$. 
		By induction,  for each $y\in K$ we can approximate the
		functions $g_{s,y}(x)=f_x^{(s)}(y)$ defined on $V_y=\mathrm{Int}(U_y\setminus C_y)$
		up to order $r-s$. By the same 
		compactness argument we obtain a natural number $m'$  and an $a$-definable family $\{D_y\}_{y\in K^n}$ of subsets $D_y\sub V_y$ with $\dim(D_y)<n$ 
		such that $g_{s,y}^{(i)}$ exists and has constant valuation on any ball
		$m'$-next to $D_y$ in $V_y$, for every multi-index, $i$, with $|i|\leq r-s$. Moreover, 
		\begin{equation}\label{taylor-2}
			| g_{s,y}(x)-\sum_{\{i:|i|<r-s\}}\frac{1}{i!}g_{s,y}^{(i)}(x_0)(x-x_0)^i |
			\leq \max_{\{i:|i|=r-s\}} |\frac{1}{i!}g_{s,y}^{(i)}(x_0)(x-x_0)^i|
		\end{equation}
		Replacing $m$ and $m'$ by their maximum, we may assume 
		$m=m'$. Define $D:=\bigcup_y D_y\times \{y\}$. By additivity of dimension $\dim(C)\le n$ and $\dim(D)\le n$. Finally,  take $E=C\cup D\cup \bigcup_y 
		(U_y\setminus V_y)\times\{y\}$. 
		Similar dimension considerations show that $\dim(E)<n+1$. 
		
		Note that for $(x,y)\in U\setminus E$ we have that, for $i$ and $s$
		such that $|i|+s\leq r$, 
		$f^{(i,s)}(x,y)$ and $g_{s,y}^{(i)}(x)$ exist and are equal.
		
		Now, for $x\in K^n$ define $W_x=\mathrm{Int}(U_x\setminus E_x)$,
		and for the functions 
		$h_{x,s,i}:y\mapsto f^{(i,s)}(x,y)$ with $s+|i|\leq r$ defined
		on $W_x$, we find a finite set $F_x\subset W_x$ such that 
		$\{F_x\}_x$ is an $a$-definable family, and there is an integer $m'$, 
		such that in every ball in $W_x$ $m'$-next to $F_x$, $h_{x,s,i}$ has constant 
		valuation. We may assume that $m'=m$ as before. Let $G$ be the closure of 
		$E\cup\bigcup_x \{x\}\times F_x\cup \bigcup_x \{x\}\times (U_x\setminus W_x)$. 
		Note that $\dim(G)<n+1$.
		
		Take $B_1\times B_2$ a ball in $U$,
		$m$-next to $G$. Then for every $x\in B_1$ we get that $B_2$ is $m$-next
		to both $C_x$ and $F_x$ and $B_2\sub W_x$. Similarly,  for every $y\in B_2$, $B_1\sub V_y$  is $m$-next to $D_y$.
		
		We conclude that for every $(x,y)\in B_1\times B_2$, $f^{(i,s)}(x,y)$
		exists and has constant valuation, for every index 
		$(i,s)$ such that $|(i,s)|\leq r$. Indeed, we have for every 
		$(x,y), (x',y')\in B_1\times B_2$ that
		\[
		|f^{(i,s)}(x',y')|=|g_{s,y'}^{(i)}(x')|=|g_{s,y'}^{(i)}(x)|=|h_{x,s,i}(y')|
		=|h_{x,s,i}(y)|=|f^{(i,s)}(x,y)|,
		\]
		as the second equality follows from the condition on $C_x$, and the fourth from those on $F_x$ and $W_x$.
		
		Now, if $(x,y)$ and $(x_0,y_0)$ are in $B_1\times B_2$, then equations
		\ref{taylor-1} and \ref{taylor-2} hold and  for the error term  of \ref{taylor-1}
		we have 
		$|f_x^{(r)}(y_0)|=|f^{(0,r)}(x,y_0)|= |f^{(0,r)}(x_0,y_0)|$.
		Denote
		\[M=\max\left \{\frac{1}{i!s!}
		|f^{(i,s)}(x_0,y_0)(x-x_0)^i(y-y_0)^s|: |(i,s)|=r\right \}.\]
		Then Equation \ref{taylor-1}  yields 
		$|f(x,y)-\sum_{s<r}\frac{1}{s!}f^{(0,s)}(x,y_0)(y-y_0)^s|\leq M$.
		Also from Equation \ref{taylor-2} we have that
		\[\left |\frac{1}{s!}f^{(0,s)}(x,y_0)(y-y_0)^s-
		\sum_{\{i:|i|<r-s\}}\frac{1}{i!s!}f^{(i,s)}(x_0,y_0)(x-x_0)^i(y-y_0)^s \right |
		\leq M.\]
		Taking the sum over $s$ smaller than $r$ and using the ultrametric inequality
		we obtain
		\[\left |\sum_s \frac{1}{s!}f^{(0,s)}(x,y_0)(y-y_0)^s-
		\sum_{\{(i,s):|i|+s<r\}}\frac{1}{i!s!}f^{(i,s)}(x_0,y_0)(x-x_0)^i(y-y_0)^s\right |\leq M.\]
		Summing this with  Equation \ref{taylor-1}  and using the 
		ultrametric inequality once more we conclude.
	\end{proof}
	
	As a consequence of the previous theorem  we obtain that partial derivatives of definable functions commute generically.
	\begin{corollary}\label{partial-derivatives-commute}
		Suppose $f:U\to K$ is a definable function for some open $U\subset K\times K$. 
		Then there exists a  open dense $U'\subset U$ such that
		for every $(x,y)\in U'$
		\[\frac{\partial}{\partial x}\frac{\partial}{\partial y}f(x,y)=
		\frac{\partial}{\partial y}\frac{\partial}{\partial x}f(x,y)\]
		and, in particular,  the terms of the above equation exist in $U'$.
		
		Moreover, if $f:U\to K$ is such that the partial derivatives 
		$\frac{\partial}{\partial x}\frac{\partial}{\partial y}f(x,y)$, 
		$\frac{\partial}{\partial y}\frac{\partial}{\partial x}f(x,y)$ exist and 
		are continuous in $U$, then they are equal.
	\end{corollary}
	\begin{proof}
		Take a $1$-dimensional closed $C\subset K\times K$ and $m$ an integer 
		as provided by the  Taylor approximation property for errors of order $3$. We may also
		assume that $\pi(C)$ and $m$ satisfy the same 
		Taylor approximation property for the function $f\pi$, where $\pi$ is the 
		coordinate permutation $(x,y)\mapsto (y,x)$.
		
		Then for $(x,y),(x_0,y_0)\in B_1\times B_2$ in a ball $m$-next to $C$
		we obtain (see Definition \ref{big-o} for the big-$O$ notation)  that
		\begin{align*}
			f(x,y)=& f(x_0,y_0)+
			(x-x_0)^2\frac{1}{2}\frac{\partial^2}{\partial x^2}f(x_0,y_0)+
			(y-y_0)^2\frac{1}{2}\frac{\partial^2}{\partial y^2}f(x_0,y_0)+ \\
			& (x-x_0)(y-y_0)\frac{\partial}{\partial x}\frac{\partial}{\partial y}f(x_0,y_0)+
			O((x-x_0,y-y_0)^3).
		\end{align*}
		Similarly, 
		\begin{align*}
			f\pi(y,x)=f(x,y)=& f(x_0,y_0)+
			(x-x_0)^2\frac{1}{2}\frac{\partial^2}{\partial x^2}f(x_0,y_0)+
			(y-y_0)^2\frac{1}{2}\frac{\partial^2}{\partial y^2}f(x_0,y_0)+ \\ & 
			(x-x_0)(y-y_0)\frac{\partial}{\partial y}\frac{\partial}{\partial x}f(x_0,y_0)+
			O((x-x_0,y-y_0)^3)
		\end{align*}
		Taking the difference we obtain
		\begin{center}
			$(x-x_0)(y-y_0)\frac{\partial}{\partial y}\frac{\partial}{\partial x}f(x_0,y_0)-
			(x-x_0)(y-y_0)\frac{\partial}{\partial x}\frac{\partial}{\partial y}f(x_0,y_0)=
			O((x-x_0,y-y_0)^3)$
		\end{center}
		Taking $h=(x-x_0)=(y-y_0)$ small we get,
		$h^2
		(\frac{\partial}{\partial y}\frac{\partial}{\partial x}f(x_0,y_0)-
		\frac{\partial}{\partial x}\frac{\partial}{\partial y}f(x_0,y_0))=
		O(h^3)$,
		so
		$\frac{\partial}{\partial y}\frac{\partial}{\partial x}f(x_0,y_0)-
		\frac{\partial}{\partial x}\frac{\partial}{\partial y}f(x_0,y_0)=O(h)$.
		This is only possible when the left-hand side is $0$, as desired.
	\end{proof}
	The following notation is intended to look similar to the monomial
	$ax^n$ for $a\in K$ and $x\in K$, for the purpose of expressing the Taylor 
	approximation of a multivariate function.
	\begin{definition}\label{monomial}
		Let $m, n$ be positive integers, and $r\in \Nn$.
		Let $J=J(r,n)=\{j\in\mathbb{N}^n: |j|=r\}$. Let 
		$a=(a_j)_{j\in J}$ be such that $a_j\in K^m$ for all $j\in J$. Then, for $x\in K^n$  we define 
		$ax^r=\sum\limits_{j\in J}a_jx^j$, where $x^j:=\prod_{i=1}^n x^{j(i)}$. Note that $x\mapsto ax^r$ is a function
		$K^n\to K^m$.
	\end{definition}
	
	As an example, consider, in the above notation, the case $r=1$. In this case $J=\{e_1,\dots, e_n\}$ and for $j\in J$ we have $x^j=x_j$ (where $x=(x_1,\dots, x_n)$), so for $a=(a_j)_{j\in J}$ with $a_j\in K^m$ we get that $ax=A\cdot x$ where $A$ is the matrix whose $j$-th column is $a_j$. 
	
	\begin{definition}\label{tn}
		Let $U\subset K^k$ be open, $f:U\to K^m$  a function and 
		$a\in U$. We say that $f$ is $P_n$  at $a$ if it is approximable
		by polynomials of degree $n$ near $a$ in the following sense:
		there are constants $b_0,\cdots,b_n$  
		$f(a+x)=\sum_{r\leq n}b_rx^r+O(x^{n+1})$. 
	\end{definition}
	
	In view of the above example, it follows immediately from the definition that a $P_1$ function is differentiable and for the coefficient $b_1$ in the definition we may take $f'(a)$ (or, more precisely $b_1^t=f'(a)$). It follows from Lemma \ref{uniqueness-taylor} below, that -- in fact -- $b_1=f'(a)^t$ whenever $f$ is $P_n$ for any $n\ge 1$. 
	
	\begin{definition}
		Let $U\subset K^k$ be open, $f:U\to K^m$  a function and 
		$a\in U$.
		We say $f$ is $T_n$ at $a$ if there is $\gamma\in \Gamma$ such that
		for every $x,x'$ with
		$|x-a|,|x'-a|<\gamma$,  we have 
		$f(x)=\sum_{r\leq n}c_r(x')(x-x')^r+O(x-x')^{n+1}$
		for $c_r$ a $P_{n-r}$ function at $a$.
	\end{definition}
	Note that in the previous definition,  
	the constant implicit in the notation  $O(x-x')^{n+1}$  (see Definition \ref{big-o})  
	does not depend
	on $x'$; so this definition requires some uniformity with respect to the center $x'$
	which is not implied by simply assuming $f$ is $P_n$ at every point of
	a ball around $a$.
	
	Note also 
	that if $f$ is $T_n$ at $c=(b,a)$ then,  in particular
	$f(z,a+x)=f(z)+f_1(z)x+\cdots+f_n(z)x^n+O(x^{n+1})$,
	for $P_{n-k}$ functions $f_k$ at $b$ (and a constant in $O(x^{n+1})$ uniform
	in $z$). This follows from the definition by taking $x=(z,a+x)$ and $x'=(z,a)$.
	
	Sums and products of $P_n$ (resp. $T_n$) functions are $P_n$ (resp. $T_n$), 
	and a vector function is $P_n$ (resp. $T_n$) if and only if its coordinate functions are $P_n$ (resp. $T_n$).
	
	We could also require the stronger condition, $ST_n$, defined  similarly $T_n$, but requiring inductively the functions $c_r$ be $ST_{n-r}$ for 
	$r=1,\dots,n$ (the base case $ST_0=T_0$).
	This definition may be
	more natural, and one can prove the same results for this notion in what 
	follows, but the stated definition is enough for Lemma \ref{xy} which is our main motivation.
	\begin{lemma}\label{uniqueness-taylor}
		If $f$ is $P_n$ at $a$, then for every number $i\leq n$, the coefficients 
		$b_i$  in Definition \ref{tn} are determined by $f$.
	\end{lemma}
	\begin{proof}
		The problem readily reduces to the case of $f$ a polynomial restricted to some open neighborhood, $U$, of the origin. I.e., We have to show that if $\sum_{i\leq n}b_ix^i=O(x^{n+1})$ in any open $U\sub K^r$ then $b_i=0$ for all $i$. 
		For $x_0$ fixed let $x=tx_0$ and consider the single variable polynomial $P(tx_0)=\sum_i\le n (b_ix_0^i)t^i=O(t^{n+1})$.
		If we knew the result for $r=1$ this would give that $b_ix_0^i=0$. Since $x_0\in U$ was arbitrary and $U$ contains
		a cartesian product of $r$ infinite sets, this implies $b_i=0$ for all $i$. So we are reduced to proving the result for $r=1$. 
		
		In this case, if $i$ is the smallest with
		$b_i\neq 0$ we get $x^i=O(x^{i+1})$ which is a contradiction.
	\end{proof}
	\begin{proposition}
		If $g$ is $P_n$ at $a$ and $f$ is $P_n$ at $g(a)$ then the composition, $f\circ g$ is $P_n$ at $a$.
		
		If $g$ is $T_n$ at $a$ and $f$ is $T_n$ at $g(a)$ then $f\circ g$ is $T_n$ at $a$.
	\end{proposition}
	\begin{proof}
		For the first statement we write
		\begin{align*}
			& f(g(a+x))= \\ 
			& f(g(a)+g_1(a)x+...+g_n(a)x^n+O(x^{n+1}))= \\
			& f(g(a))+f_1(g(a))h(a,x)+f_2(g(a))h(a,x)^2+...+O(h(a,x)^{n+1})= \\
			& fg(a)+b_1(a)x+\cdots+b_n(a)x^n+O(x^{n+1})+O(h(a,x)^{n+1}),
		\end{align*}
		where
		\begin{enumerate}
			\item $h(a,x)=g(a,x)-g(a)=g_1(a)x+\cdots+g_n(a)x^n+O(x^{n+1})$
			\item The second inequality is the application of the assumption that $f$ is $P_n$ at $g(a)$. 
			\item The coefficients $b_i$ arise by expanding the expression 
			\[
			f_k(a)h(a,x)^k=f_k(a)(g_1(a)x+\cdots+g_n(a)x^n+O(x^{n+1}))^k.
			\] 
		\end{enumerate}
		To conclude, we note that, as in the proof of Lemma  \ref{uniqueness-taylor}  $h(a,x)=O(x)$, and so 
		$O(h(a,x)^{n+1})=O(x^{n+1})$.
		
		The proof of the second statement is, essentially, similar:
		\begin{align*}
			& f(g(x))=\\
			& f(g(x')+g_1(x')(x-x')+...+g_n(x')(x-x')^n+O(x-x')^{n+1})= \\
			& f(g(x'))+f_1(g(x'))h(x,x')+f_2(g(x'))h(x,x')^2+...+f_n(g(x'))(x-x')^n+O(h(x,x')^{n+1})= \\
			& f(g(x'))+b_1(x')(x-x')+\cdots+b_n(x')(x-x')^n+O(x^{n+1})+O(h(x,x')^{n+1})    
		\end{align*}

		Where $h(x,x')=g(x)-g(x')=g_1(x')(x-x')+\cdots+g_n(x')(x-x')^n+O(x-x')^{n+1}=
		O(x-x')$, 
		and the coordinates of the coefficients $b_k(x')$ are sums
		and products of the coordinates of the 
		coefficients of $f_i(g(x'))$ and $g_j(x')$ with 
		$i,j\leq k$. By what we have just proved, those are $P_{n-i}$ functions. The constant appearing on $h(x,x')=O(x-x')$ does not depend
		on $x'$ because the $g_i$ are continuous at $a$. 
		We conclude that the $b_k$ are $P_{n-k}$ at $a$, as claimed. 
	\end{proof}
	\begin{proposition}\label{generic-regularity-0}
		Let $K$ be 1-h-minimal and $f:U\to K^m$ as definable function. Then there exists $U'\subset U$ definable
		open and dense such that $f$ is $T_n$ at every point of $U'$.
		
		In particular for every $f:U\to K^m$, there is a definable open dense subset 
		$U'\subset U$ such that $f$ is strictly differentiable in $U'$.
	\end{proposition}
	This follows from Taylor's approximation theorem (Proposition \ref{taylor0}
	in residue characteristic $0$ and Proposition \ref{taylor} in positive
	residue characteristic). The second statement follows because a $T_1$ function
	is strictly differentiable.
	
	In the next section we show that a strictly differentiable map with invertible
	derivative, definable in a 1-h-minimal valued field,
	is a local homeomorphism. Here we show that 
	the local inverse is  strictly differentiable.
	We then proceed to showing that the properties $P_n$ and $T_n$ are also preserved
	in this local inverse, though this latter fact is not used for the proof of our main results. 
	
	\begin{proposition}\label{inverse-derivative}
		Suppose $f:U\to V$ is a bijection where $U\subset K^n$ and
		$V\subset K^n$ are open. Suppose $f$ satisfies
		$|f(x)-f(y)-(x-y)|<|x-y|$ for $x,y\in U$ distinct.
		Assume $f$ is differentiable at $a$. 
		Then $f'(a)$ is invertible,
		$f^{-1}$ is 
		differentiable at $b=f(a)$
		and $(f^{-1})'(b)=f'(f^{-1}(b))^{-1}$.
		
		If $f$ is strictly differentiable at $a$, then $f^{-1}$ is strictly differentiable at $b$.
	\end{proposition}
	\begin{proof}
		Note that the hypothesis implies $|f(x)-f(x')|=|x-x'|$. This implies that  $f'(a)$ is invertible. Indeed, 
		assume otherwise, and take $x$ close to $a$ such that $f'(a)(x-a)=0$,
		to get
		$|f(x)-f(a)|<|x-a|$, a contradiction.
		
		Assume that
		$f$ is strictly differentiable at $a$.
		Take $\epsilon>0$ in $\Gamma$. Then there is $0<r\in \Gamma$ such that 
		if $|x-a|,|x'-a|<r$, then
		$|f(x)-f(x')-f'(a)(x-x')|\le \epsilon |x-x'|$.
		If we denote $y=f(x)$ and $y'=f(x')$, then we have $|y-y'|=|x-x'|$,
		so multiplying the above inequality by $f'(a)^{-1}$ we obtain
		\begin{align*}
			& |f^{-1}(y)-f^{-1}(y')-f'(a)^{-1}(y-y')|= \\
			& |f'(a)^{-1}(f'(a)(x-x')-(f(x)-f(x'))|\le \\
			& |f'(a)^{-1}||f(x)-f(x')-f'(a)(x-x')|\le  \\
			& \epsilon |f'(a)^{-1}||x-x'|=
			\epsilon |f'(a)^{-1}||y-y'|, 
		\end{align*}
		where, for a linear map, $A$,  represented by the matrix $(a_{ij})_{i,j}$  we  denote
		$|A|=\max_{i,j}|a_{ij}|$, and use the ultra-metric inequality to get $|Ax|\le |A||x|$, which we apply to obtain the first inequality in the above computation. 
		
		So we conclude that $|f^{-1}(y)-f^{-1}(y')-f'(a)^{-1}(y-y')|\leq 
		\epsilon|f'(a)|^{-1}|y-y'|$ for any $y,y'$ such that $|y-b|=|x-a|<r$ and 
		$|y'-b|=|x'-a|<r$. We have, thus,  shown that $f^{-1}$ is strictly differentiable at $b$ and 
		$(f^{-1})'(b)=f'(f^{-1}(b))^{-1}$. \\
		
		To show that $f^{-1}$ is differentiable if $f$ is substitute $x'=a$ in the above argument. 
	\end{proof}
	\begin{proposition}
		Suppose $f:U\to V$ is a bijection where $U\subset K^n$ and
		$V\subset K^n$ are open. Suppose $f$ satisfies
		$|f(x)-f(y)-(x-y)|<|x-y|$ for $x,y\in U$ distinct.
		Then if $f$ is $P_n$ (resp. $T_n$)  at $b$, $f^{-1}$ is $P_n$ (resp. $T_n$) at $f(b)$.
	\end{proposition}
	\begin{proof}
		Denote  $a=f(b)$.
		Note that the hypothesis implies than $|f(x)-f(y)|=|x-y|$ for all distinct $x,y\in U$,
		so the inverse map $f^{-1}$ is continuous, and in fact satisfies
		$|f^{-1}(x)-f^{-1}(y)|=|x-y|$ for distinct $x,y\in V$. In particular, $f^{-1}$ is 
		$T_0$ in $V$.
		
		Now, assume that $f$ is $P_n$ at $b$, with $n\ge 1$. In particular, by Proposition \ref{inverse-derivative}
		it is differentiable and  $f'(b)$ is invertible.
		
		Apply the fact that $f$ is $P_n$ (and see also the discussion following the definition) to get: 
		\[
		f(y)-f(b)=f'(b)(y-b)+f_2(b)(y-b)^2+\cdots+f_n(b)(y-b)^n+O(y-b)^{n+1}.
		\]
		Rearranging, we get:  
		\[
		y-b=f'(b)^{-1}(f(y)-f(b))-f'(b)^{-1}f_2(b)(y-b)^2-\cdots-
		f'(b)^{-1}f_n(b)(y-b)^n+O(y-b)^{n+1}.
		\]
		Putting $y=f^{-1}(x)$, and remembering $x-a=O(y-b)$ we conclude
		\[
		f^{-1}(x)-f^{-1}(a)=f'(f^{-1}(a))^{-1}(x-a)+\sum_{2\leq i\leq n}c_i(a)(f^{-1}(x)-f^{-1}(a))^i+O(x-a)^{n+1},
		\tag{$\diamond$} \]
		for some constants $c_i(a)$.
		
		Next we proceed to  showing (by induction on $k\le n$) that $f^{-1}$ is $P_k$. As $P_0$ follows from the equality $|f^{-1}(x)-f^{-1}(y)|=|x-y|$, we assume that 
		So suppose $f^{-1}$ is $P_{k-1}$. Using this, we can write
		$f^{-1}(x)-f^{-1}(a)=\sum_{1\leq j<k}b_j(a)(x-a)^j+O(x-a)^k$, 
		and apply a direct computation to obtain that 
		\[
		c_i(a)(f^{-1}(x)-f^{-1}(a))^i=
		\sum_{i\leq j\leq k}d_{ij}(a)(x-a)^j+O(x-a)^{k+1}.
		\]
		for some constants $d_{ij}$. 
		Note that as $i\ge 2$ we obtain the improved error $O(x-a)^{k+1}$.
		Substituting this in $(\diamond)$, we the conclusion follows. 
		
		Now suppose $f$ is $T_n$ at $b$. The proof in this case is similar:
		\[f^{-1}(x)-f^{-1}(x')=f'(f^{-1}(x'))^{-1}(x-x')+\sum_{2\leq i\leq n}c_i(f^{-1}(x'))(f^{-1}(x)-f^{-1}(x'))^i+O(x-x')^{n+1},\]
		so, as above,  if $f^{-1}(x)$ is $T_{k-1}$ we obtain
		\[f^{-1}(x)-f^{-1}(x')=f'(f^{-1}(x'))^{-1}(x-x')+\sum_{2\leq i\leq k}d_i(x')(x-x')^i+O(x-x')^{k+1}.\]
		Here note that
		$f'$ is $P_{n-1}$ at $b$ and so $f'(f^{-1}(x'))^{-1}$ is $P_{k-1}$ at $a$. 
		Also, following the above argument we see that
		the coordinates of $d_i(x')$ 
		are sum and products of functions of the form $b_{i'}(x')$ with
		$1\leq i'<i$, 
		for $b_{i'}(x')$ a 
		$P_{k-1-i'}$ function at $a$ (by the induction hypothesis that $f^{-1}$
		is $T_{k-1}$), and functions of the form
		$c_{i'}(f^{-1}(x'))$ for $c_{i'}(y')$ a $P_{k-i'}$ function at $b$,
		$i'\leq i$ (by the assumption
		that $f$ is $T_n$). 
		So $d_i(x')$ is a $P_{k-i}$ at $a$.
	\end{proof}
	
	The next lemma will be important in our study of the differential structure of definable groups. For the statement, recall  that $O_x$ means that the constant implicit in the notation depends
	on $x$, see Definition \ref{big-o}.
	
	\begin{lemma}\label{xy}
		Let $f:U\times V\to K^r$ be a definable function, where $U\subset K^n$
		and $V\subset K^m$ are open sets around $0$.
		Suppose $f(x,y)$ is $T_2$ at $(0,0)$, and $f(x,y)=O(x,y)^3$. 
		If 
		$axy+f(x,y)=O_x(y^2)$, then $a=0$.
	\end{lemma}
	
	\begin{proof}
		By the definition of $T_2$ (with $x=x'$, $y'=0$) we get 
		$f(x,y)=f_0(x)+f_1(x)y+O(y^2)$, for $f_1$ a $P_1$ function at $0$ and 
		$f_0$ a $P_2$ function at $0$. Fixing $x$ and expanding the Taylor polynomial of $f(x,y)$ the uniqueness
		of Taylor coefficients (Lemma \ref{uniqueness-taylor}) 
		gives, using our assumption,  $axy+f_0(x)+f_1(x)y=0$.
		
		Expanding $f_0,f_1$ around $0$ and keeping in mind $f(x,y)= O(x,y)^3$
		we get $f_0(x)= O(x^3)$ and $f_1(x)= O(x^2)$.
		Indeed, we have $f_0(x)=b_0+b_1x+b_2x^2+O(x^3)$ and 
		$f_1(x)=c_0+c_1x+O(x^2)$, so 
		$f(x,y)=b_0+b_1x+b_2x^2+c_0y+c_1xy+O(x,y)^3=O(x,y)^3$, so from the 
		uniqueness of the Taylor coefficients we get $b_0=b_1=b_2=c_0=c_1=0$.
		
		Now from $axy=O(x,y)^3$, we get $a=0$, by the uniqueness of Taylor coefficients
		again.
	\end{proof}
	
	\section{Strictly differentiable definable maps}\label{smooth-map}
	In this section we prove an inverse function theorem for definable strictly differentiable 
	maps in a 1-h-minimal valued fields. This is done adapting a standard argument from real
	analysis using Banach's fixed point theorem. In the present section we use definable spherical completeness to obtain a definable version of Banach's fixed point theorem, implying, almost formally, the desired inverse function theorem. 
	From the inverse function theorem we deduce results on the local
	structure of immersions and submersions in the usual way. We then proceed to  proving a generic
	version of the theorem on the local structure of functions of constant rank (Proposition \ref{constant-rank}). This last result
	is obtained only generically. The reason is that definable functions whose partial derivative with respect to a variable $x$ is $0$ on an open set, $U$  need not be  locally constant in $x$ in $U$. For that reason,  we give  a different argument for a weaker result, see 
	Proposition \ref{generic-relative-locally-constant}, and the discussion preceding it. \\
	
	Throughout the rest of this section, we fix an $\aleph_0$-saturated 1-h-minimal valued field $K$.  \\
	
	We start with a fixed point theorem, mentioned in  \cite[Remark 2.7.3]{hensel-min}. We first note that a version of definable spherical completeness of $1$-h-minimal fields (\cite[Lemma 2.7.1]{hensel-min}) holds in positive residue characteristic: 
	\begin{lemma}\label{spherical-completeness}
		Suppose $K$ has positive residue characteristic $p$.
		Suppose $\{B_i\}_i$ is a definable chain of open balls or 
		a definable chain of closed balls. Suppose, further,  that for every
		$i$ there is $j$ such that $\text{rad}(B_j)\leq |p|\text{rad}(B_i)$.
		Then $\bigcap_i B_i\neq \0$.
	\end{lemma}
	
	\begin{proof}
		The proof is similar to spherical completeness in residue characteristic
		$0$, see Lemma 2.7.1 of \cite{hensel-min}. It is enough consider the $1$-dimensional case, since the higher dimensional case follows by considering the coordinate projections. Note also that our assumption implies that the chain $\{B_i\}$ has no minimal element (as such an element would have valuative radius $0$). 
		
		The closed case follows from
		the open case as follows: given a definable chain  $\{B_i\}_{i\in I}$ of closed balls. For each $i$ let $r_i$ be the valuative radius of $B_i$ and let $B_i'$ be the unique open ball $B\sub B_i$ of valuative radius $r_i$ with the additional property that $B\supseteq B_j$ for all $j<i$. Obviously, $\bigcap_i B_i=\bigcap_i B_i'$ (unless the chain $B_i$ has  a minimal element, in which case there is nothing to prove). 
		
		Note that, in the above notation, the map $B_i\mapsto r_i$ is injective, so there is no harm assuming that $\{B_i\}$ is indexed by a subset of $\Gamma$. Thus, our chain $\{B_i\}$ has index set interpretable in $RV$, 
		so by \cite[Proposition 2.3.2]{hensel-minII} there is a finite set $C$ $m$-preparing the chain $\{B_i\}$ for some $m\in \Nn$. We claim that $C\cap B_i\neq \0$ for all $i\in I$. This would finish the proof, since $C$ is finite. 
		Assume, therefore, that this is not the case, and let $i_0\in I$ be such that $B_{i_0}\cap C=\0$. By assumption we can find $i<i_0$ such that $r_i<|p^m|r_{i_0}$. Then $B_i$ is a ball $m$-next to $C$, and since our chain has no minimal element, any ball $B\varsubsetneq B_i$ that is an element of our chain is not $m$-prepared by $C$, a contradiction. 
	\end{proof}

 Note that by \cite[Exanple 1.5]{BWHal} infinitely ramified $1$-h-minimal fields of positive residue characteristic need not be definably spherically model complete. Thus, the extra condition in the assumption of the above lemma is not superfluous.

	\begin{proposition}\label{fixed-point}
		Let $B_r=\{x\in K^n: |x|\leq r\}$.
		Suppose $f:B_r\to B_r$ is a definable function. 
		Assume that for distinct $x,y\in B_r$,
		we have
		\begin{enumerate}
			\item $|f(x)-f(y)|<|x-y|$ if the residue characteristic is $0$.
			\item $|f(x)-f(y)|\leq |p||x-y|$ if the residue characteristic is $p>0$.
		\end{enumerate}
		Then $f$ has a unique fixed point in $B_r$.
	\end{proposition}
	\begin{proof}
		Uniqueness is immediate from the hypothesis. For existence take the family of balls of the form $B(a)_{|f(a)-a|}$. It is a definable chain of balls indexed by $a\in B_r$.
		Indeed if $a,b\in B_r$ are distinct and the balls are disjoint then
		$|f(a)-f(b)|=|a-b|$, as the distance of points in disjoint balls does not change. Note that 
		in positive residue characteristic one has the additional hypothesis
		in Lemma \ref{spherical-completeness}, because $|f(f(a))-f(a)|\leq |p||f(a)-a|$,
		by assumption 2 on $f$.
		
		By the appropriate version of definable spherical completeness of 1-h-minimal fields (Lemma 2.7.1 of \cite{hensel-min} for residue characteristic $0$, Lemma \ref{spherical-completeness} otherwise),
		we obtain a point $x$ in the intersection of all balls. Then $x$ is a fixed point of $f$.
		Indeed, if we assume otherwise then for $y=f(x)$, we have $|f(y)-y|<|f(x)-x|$ by the hypothesis. 
		On the other hand if $a$ is arbitrary then, as $x\in B(a)_{|f(a)-a|}$, one has $|f(x)-f(a)|\leq |x-a|\leq |f(a)-a|$
		and so $|f(x)-x|\leq |f(a)-a|$. This is a contradiction and finishes the proof.
	\end{proof}
	Just as in real analysis this fixed point theorem implies an inverse function theorem.
	
	\begin{proposition}\label{bilipschitz-inverse-function}
		Suppose $f:U\to K^n$ is a definable function from an open set $U\subset K^n$ 
		satisfying the following ``bilipschitz condition'':
		for every $x,y\in U$ distinct 
		\begin{enumerate}
			\item
			$|f(x)-f(y)-(x-y)|<|x-y|$ if the residue characteristic is $0$.
			\item
			$|f(x)-f(y)-(x-y)\leq |p||x-y|$ if the residue characteristic is $p>0$.
		\end{enumerate}
		Then $f(U)$ is open and $f$ is a homeomorphism from $U$ to $f(U)$.
		If $f$ is (strictly) differentiable then $f^{-1}$ is (strictly) differentiable.
	\end{proposition}
	\begin{proof}
		Injectivity of the map follows directly from the hypothesis. The same assumptions also imply that if $x,y\in U$ are distinct then
		$|f(x)-f(y)|=|x-y|$, implying  continuity of the inverse.
		
		The main difficulty is showing that $f(U)$ is open. Translating, we may assume $0\in U$ and $f(0)=0$. We have to find
		an open neighborhood of $0$ in $f(U)$.
		Take $r>0$ such that $0\in B_r\subset U$. Then $B_{r}\subset f(U)$.
		Indeed, if $|a|\leq r$ then, by the same reasoning as above,  the function $g(x)=x+a-f(x)$ satisfies $g(B_r)\subset B_r$. By the  assumptions on $f$ this implies that $g$ satisfies the hypothesis of Proposition \ref{fixed-point}. So  $g(x_0)=x_0$ for some $x_0$, namely, $a=f(x_0)$, as claimed. 
		
		Differentiability (and strict differentiability) of $f^{-1}$  now follow from \ref{inverse-derivative}.
	\end{proof}
	
	We can finally formulate and prove the inverse function theorem for $1$-h-minimal fields: 
	
	\begin{proposition}\label{inverse-function}
		Suppose $f:U\to K^n$ is a definable function from an open set 
		$U\subset K^n$.
		Suppose $f$ is strictly differentiable at $a$ and 
		$f'(a)$ is invertible.
		Then there is an open set $V\sub U$ around $a$ such that $f(V)$ is open and $f:V\to f(V)$ is a bijection whose inverse is strictly differentiable  at $f(a)$.
	\end{proposition}
	\begin{proof}
		By the definition of strict differentiability the function $f'(a)^{-1}f$ satisfies the hypothesis of the previous proposition in a small open ball around $a$. The conclusion follows. 
	\end{proof}
	We do not know whether a definable function, $f:U\to K$, with continuous partial derivatives, but such that $f$ is not strictly differentiable in $U$ could exist\footnote{In real analysis it is well known that a function $f: U \to \Rr$ is  $\CC^1$ in $U$ if and only if it is strictly differentiable there.}. Clearly, sums, products and compositions of strictly differentiable functions are strictly differentiable, and so are locally analytic functions. Moreover, strict differentiability is first order definable, and therefore extends to elementary extensions. Also, by the generic Taylor
	approximation theorem, in the $1$-h-minimal context, any definable function in an open subset of $K^n$ is strictly differentiable in a dense open subset. See Proposition \ref{generic-regularity-0}. \\
	
	Our next goal is to study definable functions of constant rank. We first note, that without the assumption of definability, a strictly differentiable function whose derivative vanishes identically need not be locally constant:  
	\begin{example}
		Consider a function $f:\mathcal{O}\to\mathcal{O}$ 
		that is locally constant in $B\setminus \{0\}$ but near $0$ it grows like $x^2$. Such a function $f$ will be strictly differentiable, with $f'\equiv 0$, but $f$ is not locally constant at $0$. 
	\end{example}
	
	Roughly, a function as in the above example involves an infinite number of choices, so it is not definable. In contrast we have:
	
	\begin{proposition}\label{locally-constant}
		Let $f:U\to K^m$ be a function definable in an open set $U\subset K^n$.
		Assume $f$ is continuous. Assume $f$ is
		differentiable with derivative $0$ on an open
		dense subset of $U$.
		Then $f$ is locally constant with finite image.
	\end{proposition}
	\begin{proof}
		We proceed by induction on $n$,  the dimension of the domain.
		We may assume $m=1$.
		
		First assume $n=1$. By the valuative Jacobian property in 
		\cite[Corollary 3.1.6]{hensel-min} and
		\cite[Corollary 3.1.3]{hensel-minII}, there is a finite set 
		$C\subset U$ such that in $U\setminus C$ the function $f$ is locally constant.
		This implies that the fibers of $f|_{U\setminus C}$ are of dimension 1, and so
		the image of $f$ is finite. As $f$ is continuous it is locally constant, the fibers form a finite partition of
		closed and so open sets on which $f$ is constant.
		
		Now assume the proposition is valid for $n$, and suppose 
		$U\subset K^n\times K$.
		We denote $\pi_1:K^n\times K\to K^n$ the projection onto the first factor and 
		$\pi_2:K^n\times K\to K$ the projection onto the second factor. 
		
		Let $V\subset U$ be an open dense subset such that $f$ is differentiable at $V$
		with derivative $0$. Denote
		$C=U\setminus V$.
		Let $T=\{x\in K^n\mid \dim(\pi_1^{-1}(x)\cap C)=1\}$ then
		$\dim(T)<n$.
		We conclude that there is an open dense set $W\subset K^{n}$ such that
		$\dim(\pi_1^{-1}(x)\cap C)=0$ for all $x\in W$. 
		Similarly, there is an open dense set
		$P\subset K$ such that $\dim(\pi_2^{-1}(x)\cap C)<n$  for every 
		$x\in P$. Shrinking $V$ to $V\cap \pi_1^{-1}(W)\cap \pi_2^{-1}(P)$ we may assume
		that if $(x,y)\in V$ then $V_x$ is an open dense subset of $U_x$ and 
		$V_y$ is an open dense subset of $U_y$.
		By the induction hypothesis and the $n=1$ case, we conclude that 
		$f_x$ and $f_y$ are locally constant. This implies that the fiber
		$f^{-1}f(x,y)$ has dimension $n+1$. Indeed, if $B$ is an open neighborhood of $y$ on which
		$f_x$ is constant and for each $y'\in B$ we take 
		$R_{y'}\subset U_{x}$ an open 
		neighborhood of $x$ on which $f_{y'}$ is constant, then $f$ is constant
		on $\bigcup_{y'\in B} R_{y'}\times\{y'\}$.By dimension considerations, we conclude that
		the image $f(V)$ is finite. As $f$ is continuous, $f^{-1}f(V)$ is closed
		in $U$, and as $V$ is dense in $U$, we conclude $f(U)=f(V)$ is finite.
		As $f$ is continuous, we conclude $f$ is locally constant as before.
	\end{proof}
	
	Given the previous proposition we may expect that a definable strictly
	differentiable function $f:U\to K^m$ with open domain $U\subset K^r\times K^s$ and
	satisfying that $D_yf(x,y)=0$, is locally of the form $f(x,y)=g(x)$. Unfortunately
	this is not true.
	\begin{example}
		Take $f:\mathcal{O}\times \mathcal{O}\to \mathcal{O}$ defined by
		$f(x,y)=0$ if $|y|> |x|$ and $f(x,y)=x^2$ if $|y|\leq |x|$. 
		Then $f$ is strictly differentiable, $f(x,\cdot)$ is locally constant, but
		$f$ is not of the form $g(x)$ near $(0,0)$.
	\end{example}
	It is due to this pathology that the conclusion of 
	Proposition \ref{constant-rank} below only holds  generically. \\
	
	Below, we let
	$D_yf(x,y)$ be the differential of the function $f_x$, given by 
	$f_x(y)=f(x,y)$;
	we call this the derivative of $f$ with respect to $y$ (where $y$ can be a tuple of variables).
	
	\begin{proposition}\label{generic-relative-locally-constant}
		Suppose $U\subset K^n$, 
		and $V\subset K^r$ are open and $f:U\times V\to K^m$ is a definable function 
		such that $f$ is continuous
		and $D_yf=0$ on a dense open subset of $\dom(f)$. Then there exists
		an open dense set $U'\subset U$ such that $f|_{U'}$ is locally of the form $g(x)$.
	\end{proposition}
	\begin{proof}
		The set, $D$ of points $x\in U$ such that for every point of $\{x\}\times V$
		$f$ is locally of the form $g(x)$ is definable.  More precisely, $x\in D$  exactly when for all 
		$y\in V$, there exists an open ball $B\ni (x,y)$, such that for all $(x',y'), (x',y'')\in B$  we have $f(x',y')=f(x',y'')$. Thus, the statement that $D$ has dense
		interior in $U$ is a first order expressible property, so we may assume
		that $\acl=\dcl$, 
		see Fact \ref{acl=dcl} and the subsequent remark.
		
		In the course of the proof we may replace $U$ by a dense open subset
		a finite number of times. Fix, $W\subset U\times V$, a dense open set
		where $f$ is differentiable and its derivative with respect to $y$ is $0$. 
		Shrinking $U$ we may assume $W_x\subset V$ is dense open for all $x\in U$.
		By Proposition \ref{locally-constant} we know that
		$f_x$ is locally constant with finite image for every $x\in U$ (recall that $f_x(y):=f(x,y)$). The sets 
		$\mathrm{Im}(f_x)$ form a definable family of finite sets indexed by $x\in U$, so
		there is uniform bound,  $n$, on their cardinalities. 
		Denoting  $A_k=\{x\in U: |\mathrm{Im}(f_x)|=k\}$ we have  $U=A_1\cup\cdots\cup A_n$, so the union
		of the interiors of the  $A_k$ form a dense open subset of $U$. Since the closures of the $\mathrm{Int}(A_k)$ are pairwise disjoint, we may assume that $|\mathrm{Im}(f_x)|=k$ for all
		$x$ and some fixed $k$.
		Since we assumed that $\acl=\dcl$, there are definable
		functions $r_1,\dots, r_k:U\to K^m$ such that 
		$\{r_1(x),\dots,r_k(x)\}=\mathrm{Im}(f_x)$. By generic continuity of definable
		functions (Proposition \ref{generic-continuity})
		we may assume that $r_i$ are all continuous.
		Then the sets $B_i=\{(x,y): f(x,y)=r_i(x)\}$ form a finite partition of 
		$U\times V$ into closed, and so open subsets.
	\end{proof}

The next two results, describing the local structure of definable maps of full rank, are standard applications of the inverse function theorem: 
 
	\begin{proposition}\label{immersion}
		Suppose $U\subset K^k$ is a definable open set and $f:U\to K^k\times K^r$
		is a definable, strictly differentiable map. Suppose that for some $a\in U$ the derivative $f'(a)$ has full rank.
		Then there is a ball $a\in B\subset U$, 
		a ball $B_2\ni 0$, a definable open set 
		$V\subset K^k\times K^r$, and a definable strict diffeomorphism
		$\phi:V\to B\times B_2$ such that $f(B)\subset V$ and the composition
		$\phi f:B\to B\times B_2$ 
		is the inclusion
		$b\mapsto (b,0)$. 
	\end{proposition}
	\begin{proof}
		After a coordinate permutation in the target we may assume the principal
		$k\times k$ minor of $f'(a)$ is invertible.
		Consider the function $g:U\times K^r\to K^k\times K^r$ defined as
		$g(x,y)=f(x)+(0,y)$. Then $g$ is strictly differentiable and has
		invertible derivative at $(a,0)$ so by the inverse function theorem, 
		Proposition \ref{inverse-function}, 
		we can find a ball $B$ around $a$ and a ball
		$B_2$ around $0$, and open set
		$f(a)\in V$ such that $g$ restrict to a strict diffeomorphism
		$g:B\times B_2\to V$.
		If $i:B\to B\times B_2$ is the inclusion $i(b)=(b,0)$
		then we get that $gi=f$, so we conclude the statement is valid with
		$\phi=g^{-1}$.
	\end{proof}
	\begin{proposition}\label{submersion}
		Suppose $U\subset K^k\times K^r$ is a definable 
		open set, and $f:U\to K^k$ is a definable strictly differentiable map.
		Let $a\in U$. Suppose $f'(a)$ has full rank.
		Then there exists a definable open set $a\in U'\subset U$,
		a ball $f(a)\in B$,
		a ball $B_2\sub K^r$, and a definable strict diffeomorphism
		$\phi:B\times B_2\to  U'$, such that $f(U')\subset B$ and the composition 
		$f\phi:B\times B_2\to B$
		is the projection $(b,c)\mapsto b$.
	\end{proposition}
	\begin{proof}
		After applying a coordinate permutation to $U$ we may assume 
		that the principal $k\times k$ minor of $f'(a)$ is invertible.
		
		Consider the function $g:U\to K^k\times K^r$ defined as $g(x,y)=(f(x,y),y)$.
		Then $g$ is strictly differentiable with invertible differential,
		so by the inverse function theorem,
		Proposition \ref{inverse-function},
		there is an open set $a\in U'\subset U$ such that $g(U')$ is open and
		$g:U'\to g(U')$ is a strict diffeomorphism.
		
		Making $U'$ smaller we may assume 
		$g(U')=B\times B_2$ is a product of two balls.
		Then if $p:B\times B_2\to B$ is the projection $p(b,c)=b$, we get that
		$pg=f$ and so the statement is valid with $\phi=g^{-1}$.
	\end{proof}
	We can finally prove our result on the local structure of definable functions of constant rank:
	\begin{proposition}\label{constant-rank}
		Let $U\subset K^k\times K^r$ and $V\subset K^k\times K^s$ be open definable
		sets and let $f:U\to V$ be a definable strictly differentiable map such that
		for all $a\in U$ the rank of $f'(a)$ is constant equal to $k$.
		Then there exist $U'\subset U$ and $V'\subset V$  definable open sets,
		such that $f(U')\subset V'$ and there are definable 
		strict diffeomorphisms
		$\phi_1:B_1\times B_2\to U'$ and $\phi_2:V'\to B_1\times B_3$, such that
		the composition
		$\phi_2f\phi_1:B_1\times B_2\to B_1\times B_3$
		is the map
		$(a,b)\mapsto (a,0)$.
	\end{proposition}
	\begin{proof}
		Take a point $(b,c)\in U$. After a coordinate permutation in $U$ and $V$ 
		we may assume
		$f'(b,c)$ has its first $k\times k$ minor invertible.
		Then by the theorem on submersions,
		Proposition \ref{submersion},
		applied to the composition of 
		$f:U\to K^k\times K^s$ with the projection
		$K^k\times K^s\to K^k$ onto the first factor, we may assume that $U$
		is of the form $B_1\times B_2$ and $f$ is of the form 
		$f(x,y)=(x,g(x,y))$.
		As $f'$ has constant rank equal to $k$ we conclude that 
		$D_yg=0$. By the Proposition \ref{generic-relative-locally-constant} we may
		assume $g(x,y)$ is of the form $g(x,y)=g(x)$ (after passing to smaller
		open balls of $B_1$ and $B_2$ not necessarily containing $(b,c)$).
		Now the function $h:B_1\to K^k\times K^s$ defined by $h(x)=(x,g(x))$
		is a definable strictly differentiable immersion so by the theorem on 
		immersions \ref{immersion} we may after shrinking $B_1$ and composing
		with a definable diffeomorphism in the target assume that $h$ is of the form 
		$h(x)=(x,0)$. This finishes the proof.
	\end{proof}
	\section{Strictly differentiable definable manifolds}\label{manifold-section}
	In this section we define definable manifolds in a 1-h-minimal field, and variants.
	These are manifolds which are covered by a finite number of definable charts,
	with compatibility functions of various kinds. 
	
	Throughout, we keep the convention that $K$ is an $\aleph_0$-saturated $1$-h-minimal field. In case $\acl_K$ is not the same as $\dcl_K$ it is better
	to take ``\'etale domains'' instead of open subsets of $K^n$ as the local model of 
	the manifold. This is because the cell decomposition,  as provided by 
	Proposition \ref{cell-decomposition},  decomposes a definable set into a finite number
	of pieces, each of which is only a finite cover of an open set, instead of an open set. We describe this notion formally below:
	
	\begin{definition}
		Let $S\subset K^m$. A definable  function
		$f:S\to K^n$ is  (topologically) étale if it is a 
		local homeomorphism. In other words, for every $x\in T$ there is a ball
		$x\in B$ such that $f(B\cap S)$ is open and the inverse map
		$f(B\cap S)\to B\cap S\to K^m$ is continuous. 
	\end{definition}
	Informally, we think of étale maps as similar to open immersions,
	and will denote such maps accordingly 
	e.g.,  $i:U\to K^n$. We now proceed to describing the differential structure of \'etale maps (or, rather, \'etale domains): 
	
	\begin{definition}
		Suppose $i:U\to K^n$ and $j:V\to K^m$ are étale maps.
		A definable function $f:U\to V$ is strictly differentiable at $x\in U$
		if there are balls $x\in B$ and $f(x)\in B'$ such that 
		$i:B\cap U\to i(B\cap U)$,
		$j:B'\cap V\to j(B'\cap V)$ are homeomorphisms onto open sets, such that 
		$f(B\cap U)\subset B'\cap V$, and the map
		$i(B\cap U)\xrightarrow{i^{-1}} B\cap U\xrightarrow{f} B'\cap V\xrightarrow{j} j(B'\cap V)$ is strictly differentiable
		at $i(x)$. In this case the derivative $f'(x)$ is defined as the derivative
		of $i(B\cap U)\to j(B'\cap V)$.
		
		The function $f:U\to V$ is called $T_k$ at $x$ if the composition 
		$i(B\cap U)\to j(B'\cap V)$ is $T_k$ at $i(x)$.
	\end{definition}
	Note that with this definition the given inclusion $U\subset K^r$ is not 
	necessarily strictly differentiable, because the local inverses 
	$i(U\cap B)\to K^r$ of the
	map $i:U\to K^n$ are only topological embeddings, so not necessarily strictly
	differentiable. \\
	
	\textbf{For the rest of this section let $\mathcal{P}$ stand for any one of the following adjectives: topological, strictly differentiable, or $T_n$}. 
	
	\begin{definition}\label{weak-manifolds-def}
		A definable weak $\mathcal{P}$-$n$-manifold is a definable set, $M$, equipped with 
		a finite number of definable injections, $\phi_i:U_i\to M$, 
		and each $U_i$ comes equipped with an étale map $r_i:U_i\to K^n$.
		We require further that the sets
		$U_{ij}:=\phi_i^{-1}(\phi_j(U_j))$
		are open in $U_i$, and that the transition maps $U_{ij}\to U_{ji}$,  
		$ \phi_j^{-1}\phi_i$ are $\mathcal{P}$-maps. We further define:
		\begin{enumerate}
			\item A definable weak $\mathcal{P}$-manifolds is a weak $\mathcal{P}$-$n$-manifolds for some $n$. 
			\item A weak $\mathcal{P}$-manifold is equipped with a topology  making the structure maps, $\phi_i$, open immersions.
			\item A morphism of definable weak $\mathcal{P}$-manifolds is a  definable function $f:M\to N$, such that for any charts $\phi_i:U_i\to M$ and $\tau_j:V_j\to N$ the set $W_{ij}=\phi_if^{-1}\tau_j(V_j)$ is open in $U_i$ and the map $W_{ij}\to V_j$ given by $x\mapsto \tau_j^{-1}f\phi_i(x)$ is a $\mathcal{P}$-map.
			\item A definable $\mathcal{P}$-$n$-manifold is a definable weak $\mathcal{P}$-$n$-manifold, where the $U_i$ are open subsets of $K^n$ (and the maps $U_i\to K^n$ are  inclusions).
			\item A morphism of definable $\mathcal{P}$-manifolds is a morphism of 
			weak definable $\mathcal{P}$-manifolds.
		\end{enumerate}
	\end{definition}
	Definable weak $K$-manifolds are, immediately from the definition, (abstract) manifolds over $K$. As such, definable differentiable weak manifold inherit the classical differential structure. For the sake of completeness we remind the relevant definitions: 
	
	\begin{definition}
		If $M$ is a definable strictly 
		differentiable weak manifold and $x\in M$,
		then the tangent space of $M$ at $x$,  $T_x(M)$, is  the disjoint union
		of
		$T_i=T_{\phi_i^{-1}(x)}(U_i)=K^n$ for 
		$(U_i,\phi_i)$ a chart around $x$, under the identification of the spaces $T_i$ and $T_j$ associated with the charts
		$U_i,U_j$ via the map $(\phi_j^{-1}\phi_i)'(\phi_i^{-1}(x))$.
		
		For a strictly differentiable definable morphism $f:M\to N$  of definable strictly differentiable weak
		manifolds, we have a 
		map of $K$-vector spaces $f'(x):T_x(M)\to T_{f(x)}(N)$ given by the differential of the map appearing in Definition \ref{weak-manifolds-def} above. 
	\end{definition}
	
	As usual, once we have a chart around a point in a weak strictly differentiable
	manifold, we get an identification of $T_x(M)$ with $K^n$, but distinct charts
	may give distinct isomorphisms.
	
	\begin{definition}
		A definable (weak) $\mathcal{P}$-Lie group is a 
		group object in the category 
		of definable (weak) $\mathcal{P}$-manifolds.
	\end{definition}
	\begin{lemma}\label{generic-regularity-etale}
		Suppose $i:U\to K^n$ and $j:V\to K^m$ are 
		étale and $f:U\to V$ is a definable map.
		Then $f$ is continuous in an open dense subset of $U$.
		
		Also $f$ is strictly differentiable and $T_k$ in an open dense subset of $U$.
	\end{lemma}
	\begin{proof}
		For the statement about continuity, note that $V$ has the subspace topology (of $V\subset K^r$) so we may assume $V=K^r$. 
		If we denote $U'$ the interior (relative to $U$) 
		of the set of points of $U$ where $f$ is 
		continuous, then in every ball $B$ where $i$ is a homeomorphism
		$i:B\cap U\to i(B\cap U)$ we get that $B\cap U\cap U'$ is 
		dense in $B\cap U$,
		by  generic continuity of definable functions. We conclude that $U'$ is dense as required.
		
		For strict differentiability and $T_k$, by the above we may assume that $f$ is continuous. Let $U'$ be the interior of the set of all points where $f:U\to V$ is strictly differentiable and $T_k$. This is a definable open set.
		By generic differentiability and generic $T_k$ property for functions
		defined on open sets, for every point  $x\in U$ there is an open ball  $B\ni x$, such that
		$U\cap B\cap U'$ is dense in $U\cap B$. Thus,  we conclude that $U'$ is dense
		in $U$.
	\end{proof}
	Note that the previous lemma implies that
	a (weak) definable topological manifold $M$
	contains an open dense subset $U\subset M$,
	which admits a structure of a (weak) definable $T_n$ manifold extending
	the given (weak) definable topological manifold
	structure. As a consequence of Proposition
	\ref{generic-regularity} below we have that this
	structure on $U$ is unique up to isomorphism
	and restriction to a definable open dense
	subset. For that reason, several of the statements below hold (essentially unaltered) for definable weak manifolds (without further assumptions on differentiability or $T_n$). For the sake of clarity of the exposition, we keep these assumptions.  
	
	\begin{proposition}\label{generic-regularity}
		If $f:M\to N$ is a definable function and $M$, $N$ are definable 
		weak $\mathcal{P}$-manifolds,
		then $f$ is a $\mathcal{P}$-map in an open dense set of $M$.
	\end{proposition}
	\begin{proof}
		Considering the charts in $M$ we may assume $M=U\to K^n$ is étale.
		Now if $(V_i,\tau_i)$ are charts for $N$, then $f^{-1}\tau_i(V_i)$ cover
		$U$, and so the union of their interiors is open dense in $U$.
		So we may assume $N=V\to K^m$ is étale. This case is Lemma
		\ref{generic-regularity-etale}.
	\end{proof}
	
	Recall that the local dimension of a definable set $X$ is defined as  
	\[
	\dim_xX=\min\{\dim(B\cap X): 
	x\in B\text{ is a definable open neighborhood of } x\text{ in } M\}.
	\] 
	The next lemma is standard: 

 %\textcolor{red}{I think that below I prefer "Definable topological group" -- no need for "Lie". But won't insist.}
	\begin{lemma}\label{local-dimension}
		Suppose $M$ is a definable topological weak manifold.
		Let $X\subset M$ be a definable subset. 
		Then $\dim(X)=\max_{x\in X}\dim_{x}(X)$.
		
		If $G$ is a definable weak topological group and $H$ is a subgroup,
		then the dimension of $H$ is the local dimension of $H$ at any point.
	\end{lemma}
	
	\begin{proof}
		If $M=U_1\cup\cdots\cup U_n$ is a covering by open sets and 
		$\phi_i:U_i\to V_i$ is a homeomorphism onto a set $V_i$, with an 
		étale map
		$V_i\to K^n$,
		then ${\rm dim}(X)=\max_i({\rm dim }(\phi_i(X\cap U_i)))$, and the local
		dimension of $X$ at $x\in X\cap U_i$ is the local dimension of 
		$\phi_i(X\cap U_i)$ at $\phi_i(x)$, so we reduce to the case $M=V$ is étale over $K^n$.
		In fact, the result is true whenever $M\subset K^m$ with the subspace topology,
		as then the local dimension of $X\subset M$ at a point $x$
		equals the local dimension
		of $X$ at $x$ in $K^m$, and so the result follows from Proposition \ref{dimension}(3).
		
		If $G$ is a definable weak topological group and $H$ is a subgroup, then
		the local dimension of $H$ at any point $h\in H$ is constant independent 
		of $h$. Indeed the left translation $L_h:G\to G$ is a definable
		homeomorphism, that sends $e$ to $h$ and satisfies $L_h(H)=H$, so
		${\rm dim}_e(H)={\rm dim}_h(H)$.
	\end{proof}
	
	\begin{proposition}\label{etale-example}
		Suppose $T\subset K^m$ is such that there is a coordinate projection
		$\pi:T\to U$ onto an open subset $U\subset K^n$ and such that the fibres of $\pi$ are finite of constant cardinality, $s$. Assume the associated map
		$f:U\to (K^{m-n})^{[s]}$ is continuous.
		Then $T\to K^n$ is étale.
	\end{proposition}
	\begin{proof}
		Let $x\in T$. Replacing $U$ by a smaller neighborhood
		around $\pi(x)$ we may assume, using  Fact \ref{topology-group-action}, that
		$f$ lifts to a continuous function $g:U\to (K^{m-n})^s$, 
		$g=(g_1,\cdots,g_s)$.
		In this case one gets that $T$ is homeomorphic to $\bigsqcup_{i=1}^sU$
		over $U$,
		via the map $(a,i)\mapsto (a,g_i(a))$.
		
	\end{proof}
	\begin{lemma}\label{glue}
		Suppose  $M=\bigcup_{i=1}^r\phi_i(U_i)$ where  $\phi_i:U_i\to M$ are definable functions, such that the $U_i$ are  definable (weak) $\mathcal{P}$-$n$-manifolds. Suppose further that for all $i,j$ the sets 
		$U_{ij}:=\phi_i^{-1}(\phi_j(U_j))$
		are open in $U_i$, and the transition 
		maps $U_{ij}\to U_{ji}$ given by 
		$x\mapsto \phi_j^{-1}\phi_i(x)$ are $\mathcal{P}$-maps 
		Then $M$ has a unique structure
		of a definable (weak) $\mathcal{P}$-$n$-manifold 
		such that $\phi_i:U_i\to M$ is an 
		open immersion.
		
	\end{lemma}
	The proof is straightforward and omitted.
	
	\begin{proposition}\label{large-equals-open-dense}
		Suppose $M$ is a definable weak topological manifold, then $X\subset M$ is large
		if and only if the interior of $X$ in $M$ is dense in $M$.
	\end{proposition}
	\begin{proof}
		Because the dimension of $M\setminus X$ is the maximum of the local dimension at its points
		by Lemma \ref{local-dimension}, 
		we conclude that both conditions are local, and so we may assume $M=U\subset K^n$ is open.
		Here the result follows from dimension theory.
	\end{proof}
	\begin{proposition}\label{definable-is-constructible-manifold}
		Suppose $M$ is a weak definable topological manifold. Suppose $X\subset M$
		is definable. Then $X$ is a finite union of locally closed definable subsets
		of $M$.
	\end{proposition}
	\begin{proof}
		There is an immediate reduction to the case where $M=U\to K^n$ is étale. In this case
		$U$ has the subspace topology $U\subset K^s$ for some $s$.
		So it is enough to prove this for $X\subset K^s$. This is a consequence
		of Proposition \ref{definable-is-constructible}.
	\end{proof}
	In case $\acl=\dcl$  a weak manifold is generically a manifold:
	\begin{proposition}\label{weak-manifold-is-generically-strong}
		Suppose $\acl=\dcl$.

		If $M$ is a definable 
		weak $\mathcal{P}$-manifold, then there is a definable
		open dense subset $U\subset M$
		which is a definable $\mathcal{P}$-manifold.
	\end{proposition}
	\begin{proof}
		There is an immediate reduction to the case in which $i:M=U\to K^n$ is étale.
		Let $r$ be an uniform bound for the cardinality of the fibers of $U$.
		In this case if we denote $X_k\subset K^n$ the set of points $x$ such that 
		$i^{-1}(x)$ has cardinality $x$, and $U_k\subset X_k$ is the interior of $X_k$,
		then $\bigcup_{k\leq r} U_k$ is open dense in in $K^n$.
		Replacing $U$ with $i^{-1}(U_k)$ we may assume that the nonempty fibers of $i$ have
		constant cardinality. From the assumption $\acl=\dcl$ we conclude that the map
		$i(U)\to (K^r)^{[s]}$ lifts to a definable map $i(U)\to (K^r)^s$.
		There is an open dense subset $V'\subset i(U)$ such that $V'\to i(U)\to (K^r)^s$ 
		is a $\mathcal{P}$-map, see Proposition \ref{generic-regularity-etale}, 
		and we conclude that $i^{-1}(V')\cong \bigsqcup_{i=1}^sV'$ 
		over $V'$, which
		is clearly a $\mathcal{P}$-manifold.
	\end{proof}
	It seems possible that in this situation a weak manifold is already a manifold, but we
	do not need this so we do not try to prove it. \\
	
	The next couple of results are not used in the main theorems, but may be
	of independent interest.
	\begin{definition}
		Suppose $M$ and $N$ are definable strictly differentiable weak manifolds,
		and $f:M\to N$ a definable strictly differentiable function.
		
		Then $f$ is called an immersion if the derivative $f'(x)$ is injective
		at all points $x\in M$.
		
		$f$ is called an embedding if $f$ is an immersion and a homeomorphism onto 
		its image.
		
		$f$ is called a submersion if the derivative $f'(x)$ is surjective for all
		$x\in M$.
	\end{definition}
	These notions have the expected properties.
	\begin{proposition}
		Suppose $f:M\to N$ is a strictly differentiable definable map of strictly differentiable definable
		weak manifolds. If $f$ is an immersion then $M$ satisfies the following
		universal property:
		For every strictly differentiable weak definable manifold $P$, 
		and $g:P\to M$, the function $g$ is strictly differentiable and definable
		if and only is $fg$ is strictly differentiable and $g$ is definable and 
		continuous.
		
		If $f$ is an embedding and $g:P\to M$ is as above, then $g$ is a strictly differentiable definable map
		if and only if $fg$ is a strictly differentiable definable map. 
	\end{proposition}
	\begin{proposition}
		If $f:M\to N$ is a surjective submersion, then a map
		$g:N\to K$ is a strictly differentiable definable function if and only if
		the composition $gf$ is strictly differentiable and definable.
	\end{proposition}
	These two properties are a consequence of the theorems on the local 
	structure of immersions and submersions, Propositions \ref{immersion} and
	\ref{submersion}. We leave the details for the interested reader to fill.\\
	
	Suppose $M$ is a definable strictly differentiable weak 
	manifold. If $M\to N$ is a 
	surjective map of sets, it
	determines at most one structure of a definable strictly differentiable 
	weak  manifold on $N$, in such a way that $M\to N$ is a submersion.
	Also, an injective map $N\to M$ determines at most one structure of 
	a strictly differentiable weak manifold on $N$, 
	in such a way that $N\to M$ 
	is an embedding. The subsets $N\subset M$ admitting such a structure
	are called submanifolds of $M$. We also get that if $N$ is a definable topological space, and $N\to M$ is a definable and continuous
	function, then there 
	is at most one structure of a strictly differentiable definable manifold on
	$N$ extending the given topology, and for which 
	$N\to M$ is an immersion. In other words the strictly differentiable weak
	manifold structure that makes $N\to M$ an embedding is determined by the set
	$N$, and the strictly differentiable weak manifold structure that makes
	$N\to M$ an immersion is determined by the topological space $N$. 
	\begin{proposition}\label{generic-immersion}
		Suppose $M,N$ are definable strictly differentiable weak manifolds,
		and let $f:M\to N$ be an injective definable map.
		Then there is a definable open dense subset $U\subset M$ such that 
		$f|_U$ is an immersion.
	\end{proposition}
	\begin{proof}
		By Proposition \ref{generic-regularity} we may assume $f$ is strictly
		differentiable. We have to show that 
		the interior of the set $\{x\in M: f'(x) \text{ is injective} \}$ is dense in $M$.
		If this is not the case,  we can find an open nonempty subset of 
		$M$ such that $f$ is not an immersion at any point. 
		
		So suppose $M$
		is an open subset of $K^n$, $N$ is an open subset of $K^m$ and $f$ is not
		an immersion at any point. For dimension reasons $n\leq m$.
		If we define $X_k$ to be the set of points $x$ of $M$ such that 
		$f'(x)$ is of rank $k$ then $X_0\cup\dots\cup X_{n-1}=M$, and so
		if $U_r$ is the interior of $X_r$ we have that $U_1\cup\cdots\cup U_{n-1}$
		is open dense in $M$. So we may assume that $f$ is of constant rank.
		This contradicts the result in Proposition \ref{constant-rank}, since the map
		$(x,y)\mapsto (x,0)$ is not injective.
	\end{proof}
 The following facts are standard, and are probably known:
	\begin{fact}\label{finite-cover-trivial}
		Suppose $X$ is a Hausdorff space and $X\to Y$ is a surjective continuous
		map, which is a local homemorphism with fibers of constant cardinality $s$.
		Then the map $t:Y\to X^{[s]}$ given by the fibers of $p$, is continuous.
	\end{fact}
	\begin{proof}
		Let $\pi:X^{s}\setminus \Delta\to X^{[s]}$ be the canonical projection.
		Take $y\in Y$ and $\{x_1,\dots,x_s\}=p^{-1}(y)=t(y)$.
		A basic open neighborhood of 
		$t(y)$ is of the form $\pi(U_1\times\dots\times U_s)$
		for $x_k\in U_k$ open and $U_k$ pairwise disjoint.
		Shrinking $U_k$ we may assume $p|_{U_k}$ is a homemorphism onto an open set. 
		If 
		$V=\bigcap_kp(U_k)$, then
		$t^{-1}(V)\subset \pi(U_1\times\dots\times U_s)$.
	\end{proof}
 
	\begin{fact}\label{finite-cover-rigid}
        Let $X,Y,Z$ be topological spaces,
		$p:X\to Z$, $q:Y\to Z$ be surjective continuous functions and $f:X\to Y$ a continuous bijection such that $qf=p$.
		Assume that $X$ and $Y$ are Hausdorff spaces, and $p:X\to Z$ has finite fibers of constant cardinality, $s$. If the map $t:Z\to X^{[s]}$, given by $x\mapsto p^{-1}(x)$ 
		is continuous then $f$ is a homeomorphism.
	\end{fact}
	\begin{proof}
		Since $f$ is continuous and bijective, we only need to show that it is open, which is a local property. Fix some $x\in X$ and $z=p(z)$. By Fact \ref{topology-group-action} the map $Z\to X^{[s]}$ lifts, locally near $z$,  to a continuous map $(l_1,\dots, l_s):Z\to X^s$. Shrinking $Z$ to this neighborhood, and reducing $X$ and $Y$ accordingly, we may assume that $X$ is homeomorphic to $\bigsqcup_{i\leq s}Z$ over $Z$,
		via the homeomorphism $F_l:\bigsqcup_{i\leq s}Z\to X$, given by 
		$(i,z)\mapsto l_i(z)$. To see that this is a homeomorphism, note that
		the image of the $i$-th-cofactor via $F_l$ is the set
		$X_i=\{x\in X\mid x=l_ip(x)\}$, which is closed in $X$ (because
		$X$ is Hausdorff). Since there are only finitely many $X_i$ (and they are pairwise disjoint) they are also
		open. Finally,  the inverse of $F_l$ restricted to $X_i$ coincides with $p$ which
		is continuous. So $f$ restricted to $F_l(Z)$ is a homeomorphism, so $f$ is open at $x$. Since $x$ was arbitrary, the conclusion follows. 
		
		Similarly, we have a homeomorphism $F_{fl}:\bigsqcup_{i\leq s}Z\to Y$,
		which is compatible with $f$ in the sense that
		$fF_l=F_{fl}$.
	\end{proof}
	
	\begin{proposition}\label{generic-embedding}
		Let $M, N$ be  strictly  differentiable weak manifolds, and 
		$f:M\to N$ be an injective definable function.
		Then there is a dense open $U\subset M$ 
		such that $f|_U$ is an embedding.
	\end{proposition}
	\begin{proof}
		By Proposition \ref{generic-immersion} we may assume that $f$ is an immersion.
		If $V_1,\dots,V_n$ is a finite open cover of $N$ and the statement is valid
		for $f:f^{-1}V_i\to V_i$, then it is also valid for $f$. So we may assume 
		$N=V\to K^m$ is étale. From the definition we have that 
		$V\subset K^d$ has the subspace topology. We  have already seen that 
		$f:M\to V$ is an immersion, so if it is a topological embedding into $K^d$,
		then it is an embedding into $V$. So we may assume $N=K^m$.
		
		Now consider $U_1\cup\dots\cup U_n=M$ a finite open cover of $M$.
		Assume $f|_{U_i}$ is an embedding. Define 
		$U_i'=\text{Int}(U_i\setminus \bigcup _{j<i}U_i)$, and 
		$U_i''=U_i'\setminus \bigcup_{j\neq i}f^{-1}\cl(f(U_j'))$.
		Note that $\bigcup_i U_i'\subset M$ is an open dense set.
		Also,
		note that $f(U_i'\setminus U_i'')= \bigcup_{j\neq i}f(U_i')
		\cap \cl(f(U_j'))\subset
		\bigcup_{j\neq i}\cl(f(U_j'))\setminus f(U_j')$. So we conclude that 
		$\dim(U_i'\setminus U_i'')<\dim(M)$, by Proposition 
		\ref{dimension-boundary-function}. Thus, replacing
		$U_i$ with $U_i''$ we may assume $\cl (f(U_i))\cap f(U_j)=\emptyset$ for
		distinct $i,j$. In this case one verifies that $f$ is a topological
		embedding. 
		
		We are thus reduced to the case where $M=U\to K^n$ is étale. 
		
		Consider for each $I\subset \{1,\dots,m\}$ of size $n$, the set $A_I$
		of $x\in U$ such that the $I$-th-minor of $f'(x)$ is invertible.
		As $U=\bigcup_IA_I$ we conclude that $U'=\bigcup_I\text{Int}(A_I)$ is open dense
		in $U$, so by the reduction in the previous paragraph we may assume that
		the composition of $f:U\to K^m$ with the projection onto the first $n$
		coordinates $p:K^m\to K^n$ is an étale immersion.
		If $s$ is a uniform bound for the size of the fibers of $U$ over $K^n$,
		then we can take $A_k$ the set of $x\in K^m$ such that the fiber 
		$(pf)^{-1}(x)$ has $k$ elements, and consider 
		$U'=\bigcup_{k\leq s}\text{Int}(A_k)$. So we may assume that if
		$V=pf(U)$, the fibers of $U$ over $V$ have the same size $s$.
		
		In this case the function $V\to U^{[s]}$, given by $x\mapsto (pf)^{-1}(x)$,
		is continuous, and the function $f$ is  topological homeomorphism
		$U\to f(U)$, see Facts \ref{finite-cover-trivial} and \ref{finite-cover-rigid}.
	\end{proof}
	
	We can now prove a $1$-h-minimal version of Sard's Lemma (compare with \cite[Theorem 2.7]{WilComp} for an analogous result in the o-minimal setting). Namely, given a definable strictly differentiable morphism, $f$, of definable weak manifolds, call a value $x$ of $f$ regular when $f$ is a submersion at every point of $f^{-1}(x)$.  The statement is, then, that the set of singular values is small:  
	\begin{proposition}\label{sard}
		Suppose $M$ and $N$ are definable strictly differentiable weak manifolds.
		If $f:M\to N$ is a strictly differentiable map, then there exists an open
		dense subset $U\subset N$ such that $f:f^{-1}(U)\to U$ is a submersion.
	\end{proposition}
	\begin{proof}
		We have to see that the image via $f$ of 
		the set of points $x\in M$ such that $f'(x)$ is not surjective,
		is nowhere dense in $N$. 
		This property is expressible by a first order formula, so we may
		assume that $\acl=\dcl$, see Fact \ref{acl=dcl} and the subsequent remark.
		
		Let $m=\dim(N)$.
		Let $X\subset M$ a definable 
		set such that for all $x\in X$,
		$f'(x)$ is not surjective. We have to see that $\dim f(X)<m$.
		We do this by induction on the dimension of $X$. 
		The base case,  when $X$ is finite, is trivial. 
		
		The dimension of $f(X)$
		is the maximum of the local dimensions at points, 
		see Proposition \ref{local-dimension}. So we may assume $N\subset K^m$ is open.
		Covering $M$ by a finite number of charts we may assume $M\to K^n$ is étale,
		say $M\subset K^r$. 
		
		Then by Proposition \ref{cell-decomposition}
		there exists a finite partition of $X$ into definable sets such that 
		if $X'$ is an element of the partition, there exists a coordinate projection
		$p:K^r\to K^l$, which restricted to $X'$ is a surjection $X'\to U$ onto an
		open subset $U\subset K^l$ with finite fibers of constant cardinality.
		If we prove that $\dim f(X')<m$ for every element $X'$ of the
		partition then also $\dim f(X)<m$.
		So we may assume there is a coordinate projection 
		$p:X\to U$ onto an open subset $U\subset K^l$, such that 
		$p^{-1}(u)$ has $t$ elements for all $u\in U$.
		From the assumption $\acl=\dcl$, we get that there are definable sections
		$s_1,\dots,s_t:U\to X$, such that $\{s_1(u),\cdots,s_t(u)\}=p^{-1}(u)$, for 
		all $u\in U$.
		As $X=\bigcup_{i\le t}s_i(U)$, we may assume $p:X\to U$ is a bijection
		with inverse $s:U\to X$.
		The map $s:U\to M$ becomes strictly differentiable in an open
		dense $V\subset U$, see Proposition \ref{generic-regularity}. 
		As $s(U\setminus V)$ has smaller dimension than $X$,
		we may assume that $s$ is strictly differentiable.
		Now note that $s$ has 
		image in $X$ and so the composition $U\to M\to N$ has derivative which is 
		not surjective at any point of $U$. 
		So we have reduced to the case in which $M=U\subset K^n$ is open
		and $X=U$. 
		
		If we consider $A_k\subset U$, the set defined by
		$A_k=\{x\in U: f'(x)\text{ has rank }k\}$,
		then $\bigcup_{k<m}A_k=U$ and so $\bigcup_{k<m}\text{Int}(A_k)$
		is open and dense in $U$. We conclude by the induction hypothesis that the 
		image of $U\setminus \bigcup_{k<m}\text{Int}(A_k)$ is nowhere dense in $N$,
		and so we may assume that $f'(x)$ has constant rank $k$ in $U$, for a $k<m$.
		
		Consider the set $Y=\{x\in U:  \dim f^{-1}f(x)\geq \dim(U)-k\}$.
		Then $Y$ is definable, because dimension is definable in definable families.
		Also, $Y$ has dense interior
		by the constant
		rank theorem Proposition \ref{constant-rank}. So once more by the induction
		hypothesis we may assume $f^{-1}f(x)$ has dimension at least 
		$\dim(U)-k$ for all $x\in U$. Then the dimension of $f(U)$ is at most
		$k$, by the additivity of dimension.
	\end{proof}
	If $M\to N$ is a map of strictly differentiable weak manifolds, and $y\in N$ is a regular
	value, then one can show $f^{-1}(y)\subset M$ is a strictly differentiable weak submanifold. 
%	When the adjective weak is removed, I do not know if this is true.
	\section{Definable Lie groups}\label{groups-section}
	In this section we show that every definable group is a definable 
	weak Lie group
	and that the germ of a definable weak Lie group morphism is determined 
	by its derivative at the identity. \\
	
	The proof of the following lemma was communicated to us by Martin Hils.
	\begin{lemma}\label{generic-in-group-0}
		Suppose $G$ is a group 
		$a$-definable in a pregeometric theory,
		and that $X,Y\subset G$ are non-empty $a$-definable sets of dimension smaller
		than $G$. 
		If $g\in G$ is such that
		$\dim(gX\cap Y)=\dim(X)$, then $\dim(g/a)\leq \dim(Y)$.
		
		In particular, there exists $g\in G$ such that 
		$\dim(gX\cap Y)<\dim(X)$
	\end{lemma}
	\begin{proof}
		Denote $d=\dim(X)$ and $d'=\dim(Y)$. 
		Suppose $\dim(gX\cap Y)=d$. Note that $d\leq d'$. 
		Let 
		$h'\in gX\cap Y$ be such that $\dim(h'/ag)=d$.
		Let $h=g^{-1}h'$. 
		
		As $h\in X$ we have 
		$d\geq \dim(h/a)\geq \dim(h/ag)=\dim(h'/ag)=d$.
		The first inequality is because $h\in X$, the third one because $h$ and 
		$h'$ are inter-definable over $ag$, and the fourth one by choice of $h'$.
		We conclude that $h$ and $g$ are algebraically independent over $a$.
		
		Then we obtain that
		$d'\geq \dim(h'/a)\geq \dim(h'/ah)=\dim(g/ah)=\dim(g/a)$.
		The first inequality because $h'\in Y$, the third equality because $h'$ and $g$
		are inter-definable over $ah$, and the fourth equality because $h$ and $g$
		are algebraically independent over $a$.
		
		For the second statement, note that
		if $g\in G$ is such that $\dim(g/a)=\dim(G)$, or more generally 
		$\dim(g/a)>d'$, then $\dim(gX\cap Y)<\dim(Y)$.
	\end{proof}
	
	The next lemma generalizes Lemma 2.4 of \cite{Pi5} for o-minimal theories. Pillay's proof can be seen to generalize, with some effort,  to geometric theories.  We give a different proof: 
	
	\begin{lemma}\label{generic-in-group}
		Suppose, $G$ is a group definable in a pregeometric theory and suppose
		$X\subset G$ is such that $\dim(G\setminus X)<\dim(G)$.
		Then a finite number of translates of $X$ cover $G$.
	\end{lemma}
	\begin{proof}
		Suppose we have $g_0,\cdots,g_n\in G$ such that 
		$\dim(G\setminus(\bigcup_k g_kX))=m$. By the Lemma \ref{generic-in-group-0} 
		applied to $G\setminus X$ and $G\setminus (\bigcup_kg_kX)$
		we get
		that there is $g_{n+1}\in G$ such that 
		$\dim(G\setminus\bigcup_{k\leq (n+1)}g_kX)<m$, which finishes the proof.
	\end{proof}
	\begin{lemma}\label{large-generates}
		Suppose $G$ is a definable group in a pregeometric theory and $V\subset G$
		is large. Then every $g\in G$ is a product of two elements in $V$.
	\end{lemma}
	\begin{proof}
		The proof of \cite[Lemma 2.1]{Pi5} works:
		if we take $h\in G$ generic over $g$, then $h^{-1}g$ is also generic over
		$g$, and so $h,h^{-1}g\in V$ and their product is $g$.
	\end{proof}

	\begin{proposition}\label{definable-is-lie}
		A definable group can be given the structure of a 
		definable strictly differentiable weak $T_k$-Lie group.
		The forgetful functor from definable strictly differentiable weak Lie groups
		to definable groups is an equivalence of categories.
		
		If $\acl=\dcl$ the forgetful functor from definable strictly differentiable
		$T_k$-Lie groups to definable groups is an equivalence of categories.
	\end{proposition}
	\begin{proof}
		That the forgetful functor 
		is full follows from Proposition \ref{generic-regularity}.
		Indeed, suppose $G$ and $H$ are strictly differentiable or $T_k$-Lie groups,
		and let $f:G\to H$ be a definable group morphism. Then by Proposition 
		\ref{generic-regularity} there is an open dense $U\subset G$ such that
		$f:U\to H$ is strictly differentiable or $T_k$. If $g_0\in U$ is arbitrary,
		and $g\in G$, consider the formula
		$f=L_{f(g)f(g_0)^{-1}}fL_{g_0g^{-1}}$, where we are denoting 
		$L_h$ the left
		translate by $h$. 
		Now, $L_{g_0g^{-1}}$ is a strict diffeomorphism
		or $T_k$-isomorphism which sends $g$ to $g_0$, and $L_{f(g)f(g_0)^{-1}}$
		is a strict diffeomorphism or $T_k$-isomorphism. We conclude that
		$f$ being 
		strictly differentiable or $T_k$ at $g_0$ implies that $f$ is strictly differentiable of $T_k$ at $g$.
		
		To see that the forgetful functor 
		is essentially surjective one follows the proof of 
		\cite[Proposition 2.5]{Pi5}.
		Namely, let $G$ be of dimension $n$. Decompose $G$
		as in Proposition \ref{cell-decomposition},
		and let $V_0\subset G$ be the union of the $n$-dimensional pieces $U_0,\cdots, U_r$.
		Give $V_0$ the structure of a weak strictly differentiable manifold with charts the 
		inclusions $U_i\to V_0$, see Proposition \ref{etale-example}.
		
		Note that $V_0^{-1}\subset G$ is large in $G$, as the inverse function  
		is a definable bijection  sending the large subset $V_0$ onto $V_0^{-1}$.
		As the intersection of two large sets is large, we conclude that 
		$V_0\cap V_0^{-1}$ is large in $G$.
		A fortiori, $V_0\cap V_0^{-1}$ is large in $V_0$ and so it contains an open
		dense subset of $V_0$, see Proposition \ref{large-equals-open-dense}.
		Let $V_1\subset V_0\cap V_0^{-1}$ be open dense in $V_0$ such that
		the inverse function on 
		$V_1$ (and into $V_0$) is strictly differentiable and $T_k$, 
		see Proposition \ref{generic-regularity}.
		
		In a similar way, we have that $V_0\times V_0\cap m^{-1}(V_0)$ is large in 
		$G\times G$ 
		Indeed, $m^{-1}(V_0)$ is the inverse image of the large subset
		$G\times V_0$ of $G\times G$ under the definable bijection $(\id,m):G\times G\to G\times G$.
		In the same way as before, we find
		$Y_0\subset V_0\times V_0\cap m^{-1}(V_0)$ open and dense in $V_0\times V_0$
		such that the multiplication
		map
		$Y_0\to V_0$ is strictly differentiable and $T_k$.
		
		Now we take
		\[
		V_1'=\{g\in V_1: (h,g),(h^{-1},hg)\in Y_0\text{ for all }h\text{ generic over }g\}.
		\]
		Note that $V_1'$ is definable, because $g\in V_1'$ is equivalent to 
		$\dim(G\setminus X_g)<n$ for $X_g=\{h\in G: (h,g),(h^{-1},hg)\in
		Y_0\}$, and dimension is definable in definable families in 
		geometric theories.
		Note also that $V_1'$ is large in $G$,
		because if $g\in G$ is generic and $h\in G$ is generic over $g$,
		then $(h,g)$ is generic in $G\times G$, and $(h^{-1},hg)$, being the image of a definable
		bijection at $(h,g)$ is also generic in $G\times G$, so they belong to $Y_0$,
		because $Y_0$ is large in $G\times G$.
		
		Now take
		$V_2$ the interior of $V_1'$ in $V_0$ and $V=V_2\cap V_2^{-1}$.
		Then $V_2$ is large in $V_0$ by Proposition \ref{large-equals-open-dense},
		and so it is also large in $G$.
		So we conclude that $V$ is an open dense subset of $V_0$.
		
		Define also $Y=\{(g,h): g,h, gh\in V, (g,h)\in Y_0\}$, then $Y$ is open dense in 
		$V_0\times V_0$.
		This is because $Y$ is large in $G\times G$, 
		with arguments as above, and it is open in $Y_0$,
		because multiplication is continuous in $Y_0$.
		
		Then we have shown:
		\begin{enumerate}
			\item $V$ is large in $G$.
			\item $Y$ is dense open subset of $V\times V$,
			and multiplication $Y\to V$ is strictly differentiable and $T_k$.
			\item Inversion is a strictly differentiable $T_k$-map from $V$ onto $V$.
			\item If $g\in V$ and $h\in G$ is generic in $G$
			over $g$ then $(h,g),(h^{-1},hg)\in Y$
		\end{enumerate}
		For the last item, note that $h,hg,h^{-1}\in G$
		are generic, and so they belong to $V$. Also, because $g\in V_1'$, one has that 
		$(h,g), (h^{-1},hg)\in Y_0$.
		
		From this one gets 
		\begin{enumerate}[(a)]
			
			\item
			For every $g,h\in G$ the set $Z=\{x\in V: gxh\in V\}$ is open
			and $Z\to V$ given by $x\mapsto gxh$ is strictly differentiable and $T_k$.
			\item
			For every $g,h\in G$, the set
			$W=\{(x,y)\in V\times V: gxhy\in V\}$ is open in $V\times V$ and the map
			$W\to V$ given by $(x,y)\mapsto gxhy$ is strictly differentiable and $T_k$.
		\end{enumerate}
		Indeed, for (a), assume $x_0\in Z$, take $h_1$ generic over $h$ and
		$k$ generic over $g,x,h,h_1$. Take $h_2=h_1^{-1}h$.
		Note that $h_1,h_2\in V$.
		Now one writes
		$f(x)=gxh$ as a composition of strictly differentiable and $T_k$
		functions defined on 
		an open neighborhood of $x_0$ in the following way.
		Consider the set
		$Z_1=\{x\in V:(kg,x)\in Y, (kgx,h_1)\in Y, (kgxh_1,h_2)\in Y, (k^{-1},kgxh)\in Y\}$,
		then by item 2 we have that $Z_1$ is open and the
		map $x\mapsto gxh=k^{-1}(((kgx)h_1)h_2)$ is a composition of strictly differentiable and
		$T_k$ functions.
		Also $x_0\in Z_1$ by item 4.
		
		Similarly for (b) given $(x_0,y_0)\in W$ the set 
		\[
		W_1=\{(x,y)\in V: (kg,x),(kgx,h_1), (kgxh_1,h_2),
		(kgxh,y), (k^{-1},kgxhy)\in Y\}
		\]
		is open by item (3), contains $(x_0,y_0)$ 
		by item (4) and in 
		$W_1$ the required map is a composition of strictly differentiable 
		and $T_k$ functions.
		
		By (1) above and Lemma \ref{generic-in-group} a finite number of translates, $g_0V,\dots, g_nV$, cover $G$. 
		Consider 
		the maps $\phi_i:V\to G$ given by $\phi_i(x)=g_ix$.
		It is straightforward to verify, using (a), (b) and (3) above 
		that these charts endow  $G$ with a (unique) structure of a strictly differentiable or $T_k$ manifold, as in Lemma \ref{glue}, and with this structure $G$ is a Lie group.
		
		For example, to see that the transition maps are strictly differentiable
		or $T_k$, we have to see that the sets $V\cap g_i^{-1}g_jV$ are open,
		and the maps $\phi_{i,j}: V\cap g_i^{-1}g_jV\to V\cap g_j^{-1}g_iV$ given by 
		$x\mapsto g_j^{-1}g_ix$ are strictly
		differentiable or $T_k$. This is a particular case of (a). Similarly, (b)
		translates into the multiplication being strictly differentiable or $T_k$
		and (a) and (3) translate into the inversion being strictly differentiable and 
		$T_k$.
		
		When $\acl=\dcl$ an appropriate version of cell decomposition in 
		Proposition \ref{cell-decomposition} gives the result by repeating the above
		proof. Alternatively, we can see it directly from the result we have just proved 
		and 
		Proposition \ref{weak-manifold-is-generically-strong} and 
		Lemma \ref{generic-in-group}
		(and the appropriate version of the Lemma \ref{glue}).
		Indeed, if $G$ is a definable group in a 1-h-minimal field with $\acl=\dcl$,
		then $G$ has the structure of a weak strictly differentiable or $T_k$-Lie 
		group. By Proposition \ref{weak-manifold-is-generically-strong} there is an
		open dense $U\subset G$ such that $U$ is a strictly differentiable or 
		$T_k$-manifold. By Proposition \ref{large-equals-open-dense} and Lemma 
		\ref{generic-in-group} a finite number of translates of $U$
		cover $G$, $g_1U\cup\cdots\cup g_nU=G$. Then the functions 
		$\phi_i:U\to G$ given by $x\mapsto g_ix$ form a gluing data for $G$ which
		makes it a strictly differentiable or $T_k$-Lie group.
	\end{proof}
	As the previous result implies that every definable group $G$ admits a structure
	of a definable weak Lie group which is unique up to a unique isomorphism,
	whenever we mention a property of the weak Lie group structure
	we understand it  with respect to this structure.
	\begin{definition}
		A definable strictly differentiable local Lie group is given by a definable open
		set containing a distinguished point 
		$e\in U\subset K^n$, a definable open subset $e\in U_1\subset U$,
		and definable strictly differentiable maps $U_1\times U_1\to U$ denoted as
		$(a,b)\mapsto a\cdot b$ and $U_1\to U$ denoted as $a\mapsto a^{-1}$, such that
		there exists $e\in U_2\subset U_1$ 
		definable open such that
		\begin{itemize}
			\item $a\cdot e=e\cdot a=a$ for $a\in U_2$.
			\item If $a,b,c\in U_2$ then $a\cdot b\in U_1, b\cdot c\in U_1$ and $(a\cdot b)\cdot c=a\cdot (b\cdot c)$.
			\item If $a\in U_2$ then $a^{-1}\in U_1$ and $a\cdot a^{-1}=a^{-1}\cdot a=e$.
		\end{itemize}
		
		Given two definable strictly differentiable local Lie groups, $U$ and $V$,
		a definable strictly differentiable local Lie group morphism
		is given by a definable strictly differentiable map $f:U'\to V_1$ for a 
		$e\in U'\subset U_1$ open, with $U_1$ and $V_1$ as in the above definition,
		and such that $f(e)=e$, $f(a\cdot b)=f(a)\cdot f(b)$ and 
		$f(a^{-1})=f(a)^{-1}$ for
		$a\in U'$.
		Also two such maps $f_1$ and $f_2$ are identified as morphisms if 
		they have the same germ around $0$, in other words, if there is a definable
		open neighborhood of the identity
		$W\subset \dom(f_1)\cap \dom(f_2)$ such that 
		$f_1|_W=f_2|_W$.
	\end{definition}
	It is common to only consider local groups where $e=0$, and translating we 
	see that every local group is isomorphic to one with this condition.
	In this case we denote the distinguished element by $e$ whenever we emphasize
	its role as a local group identity.
	
	We will usually identify a local group with its germ at $e$. In those terms,  the prototypical example of a local Lie group is the germ around the identity
	of a Lie group. \\
	
	The following fact is a well known application of the chain rule. We give the short proof for completeness: 
	
	\begin{fact}\label{m'}
		Suppose $U$ is a local definable strictly differentiable Lie group.
		Then the multiplication map $m:U_1 \times U_1\to U_0$ has derivative the 
		$m'(0)(u,v)=u+v$.
		The inverse $i:U_1\to U_0$ has derivative $i'(0)(x)=-x$.
		The $n$-power $p_n:U_n\to U_0$ has derivative $p_n'(0)(x)=nx$
	\end{fact}
	\begin{proof}
		The formula for $m'(0)$
		follows formally from the equations $m(x,0)=x, m(0,y)=y$.
		Indeed if $m(x,y)=ax+by+o(x,y)$, then plugging $y=0$ we obtain $a=1$ and plugging $x=0$ we obtain $b=1$. Here we are using the small $o$ notation, $f=o(x,y)$,
		meaning that for all $\epsilon>0$ there is an $r$ such that if $|(x,y)|<r$ then
		$|f(x,y)|\leq \epsilon|(x,y)|$, and we are using the uniqueness of 
		derivatives for the strictly differentiable functions $m(x,0)$ and $m(0,y)$.
		
		From this the formula for $i'(0)$ follows from $m(x,i(x))=0$ and the chain 
		rule.
		
		The formula for $p_n$ follows inductively from the chain rule and 
		$p_n(x)=m(p_{n-1}(x),x)$.
	\end{proof}
	We  give some results on subgroups and quotient groups 
 %\textcolor{red}{Are there results on quotient groups?}
	These are not
	needed for the main applications.
	\begin{proposition}
		Suppose $f:G\to H$ is a surjective definable group morphism.
		Then $f$ is a submersion.
	\end{proposition}
	\begin{proof}
		This is a consequence of Sard's Lemma, Proposition \ref{sard}.
	\end{proof}
	\begin{fact}\label{constructible-nowhere-dense}
		Suppose $X$ is a topological space and $Y\subset X$ is a finite union of locally closed
		subsets of $X$. Then every open nonempty subset of $X$ contains an open nonempty subset
		which is disjoint from $Y$ or contained in $Y$.
	\end{fact}
	\begin{proof}
		The property mentioned is closed under Boolean combinations and is true for open subsets.
	\end{proof}
	\begin{proposition}\label{subgroups-are-closed}
		Suppose $G$ is a definable group and $H\subset G$ is a definable subgroup.
		Then $H$ is closed in $G$.
	\end{proposition}
	\begin{proof}
		Recall that $H$ is a finite union of locally closed subsets of $G$,
		see for instance Proposition \ref{definable-is-constructible-manifold}.
		So by applying Fact \ref{constructible-nowhere-dense} to $H\subset \bar{H}$,
		we conclude that $H$ has nonempty relative interior in $\bar{H}$.
		As $H$ is a subgroup we conclude by translation that $H$ is open in $\bar{H}$.
		An open subgroup is the complement of some of its translates, so it is also
		closed. We conclude that $H=\bar{H}$ is closed.
	\end{proof}
	\begin{proposition}\label{subgroups-are-submanifolds}
		Suppose $H\subset G$ is a subgroup of $G$. Then with the structure of 
		weak definable strictly differentiable manifolds on $G$ and $H$, the inclusion
		$i:H\to G$ is a closed embedding.
	\end{proposition}
	\begin{proof}
		By Proposition \ref{generic-embedding},
		there is an open dense set $U\subset H$ such that $i|_{U}$ is an embedding.
		Replacing $U$ by $U'=U\setminus \cl(i(H)\setminus i(U))$ if necessary,
		and keeping in mind Proposition \ref{dimension-boundary-function} to show 
		$U'$ is large in $H$,
		we may assume $i(U)$ is open in $i(H)$.
		By translation we conclude that $i$ is an immersion. Also for an open set
		$V\subset H$ we have that $i(V)=i(\bigcup_{h\in H} hU\cap V)=\bigcup_hhi(U\cap h^{-1}V)$ is open in $i(H)$. Since $i$ is injective, the conclusion follows. 
	\end{proof}
	
	As a consequence of the theorem on constant rank functions, Proposition 
	\ref{constant-rank}, we have the following result:
	\begin{corollary}\label{dimension-of-kernel}
		Suppose $U$ and $V$ are definable strictly differentiable local Lie groups
		and let $g,f:U\to V$ be definable strictly differentiable local Lie group
		morphisms. If we denote $Z=\{x\in U: g(x)=f(x)\}$, then 
		$\dim_e Z=\dim(\ker (f'(e)-g'(e)))$. 
		
		In particular if $G$ and $H$ are definable 
		strictly differentiable weak Lie groups
		and $g,f$ are definable strictly differentiable Lie group morphisms then 
		$\dim\{x: f(x)=g(x)\}=\dim(\ker(f'(e)-g'(e)))$.
	\end{corollary}
	\begin{proof}
		The second result follows from the first because of Lemma 
		\ref{local-dimension}.
		
		In order to keep the proof readable we only verify the first statement in the case
		of weak Lie groups. The proof for local Lie groups is similar. 
		By translating in $G$ we see that the map $f\cdot g^{-1}:G\to G$
		has, at any point of $G$, derivatives of constant rank 
		equal to ${\rm dim}(G)-k$, for $k=\dim(\text{ker}(f'(e)-g'(e)))$.
		Indeed, if $u\in G$, then
		$(f\cdot g^{-1})L_u=L_{f(u)}R_{g(u)^{-1}}(f\cdot g^{-1})$, 
		where 
		$L_u$, $R_u$ denote the  left and right translates by $u$, respectively.
		By the chain rule we get $(f\cdot g^{-1})'(u)L_u'(e)=(L_{f(u)}R_{g(u)^{-1}})'(f(u)\cdot g(u)^{-1})(f\cdot g^{-1})'(e)$. As $L_u$ and $L_{f(u)}R_{g(u)}$ are definable strict diffeomorphisms, 
		we have that their derivatives at any point 
		are vector space isomorphisms, so we conclude that the rank of 
		$(f\cdot g^{-1})'(u)$
		equals the rank of $(f\cdot g^{-1})'(e)=f'(e)-g'(e)$ (see Fact \ref{m'}), as desired.
		
		By the theorem on constant rank functions, 
		Proposition \ref{constant-rank},
		we conclude that there are nonempty open sets $U\subset G$ and 
		$V\subset H$, balls $B_1, B_2$ and $B_3$ around the origin
		and 
		definable strictly differentiable isomorphisms
		$\phi_1:U\to B_1\times B_2$ and $\phi_2:V\to B_1\times B_3$, such that
		$f(U)\subset V$ and $\phi_2f=\phi_1\alpha$ for 
		$\alpha:B_1\times B_3\to B_1\times B_3$ the function $(x,y)\mapsto (x,0)$. 
		Translating in $G$
		we may assume $e\in U$. More precisely, from the formula
		$(f\cdot g^{-1})L_u=(L_{f(u)}R_{g(u)^{-1}})(f\cdot g^{-1})$ discussed before,
		if $u\in U$ maps to $(0,0)$ under $\phi_1$, 
		then $e\in u^{-1}U$, so we may replace $(U,V,\phi_1,\phi_2)$
		by $(u^{-1}U, f(u)^{-1}Vg(u), \phi_1L_u, \phi_2L_{f(u)}R_{g(u)}^{-1})$.
		Note also that $\phi_1(e)=(0,0)$.
		
		In this case we obtain
		$\{0\}\times B_2=\phi_1(Z\cap U)$ so the local dimension of $e$ at $Z$ is
		the local dimension of $\{0\}\times B_2\subset B_1\times B_2$ at $(0,0)$,
		which is the dimension of $B_2$ and is as in the statement.
	\end{proof}

	In the particular case the dimension 
	$\dim(\text{ker}(f'(e)-g'(e)))$ of the previous statement 
	equals the dimension of $G$ we get: 
	\begin{corollary}\label{map-germ-at-1}
		Suppose $U$ and $V$ are definable strictly differentiable local Lie groups
		and let $g,f:U\to V$ be definable strictly differentiable local Lie group
		morphisms. 
		Then $f$ and $g$ are equal (as local Lie group morphisms) if and
		only if $f'(0)=g'(0)$.
		
		In particular if $G$ and $H$ are definable strictly differentiable weak Lie groups
		and $g,f$ are definable strictly differentiable Lie group morphisms then 
		$f$ and $g$ coincide in an open neighborhood of the identity $e$ if and
		only if $f'(e)=g'(e)$.
	\end{corollary}
	The following two corollaries are not needed for the sequel, but may be interesting
	on their own right.
	\begin{corollary}\label{intersection-tangent}
		Suppose $H_1$ and $H_2$ are subgroups of the strictly differentiable definable 
		weak Lie
		group $G$. Then $T_e(H_1\cap H_2)=T_e(H_1)\cap T_e(H_2)$ as subspaces of 
		$T_e(G)$.
	\end{corollary}
	\begin{proof}
		This is a consequence of Corollary \ref{dimension-of-kernel}.
		Indeed, we know $H_1$, $H_2$ and $H_1\cap H_2$ are strictly differentiable
		definable weak Lie groups and the inclusion maps $H_1\cap H_2\to H_i$ and 
		$H_i\to G$ are strictly differentiable immersions, for example by
		Proposition \ref{subgroups-are-submanifolds},
		so the statement makes sense.
		We also have the diagonal map $\Delta:H_1\cap H_2\to H_1\times H_2$
		is the equalizer of the two projections $p_1:H_1\times H_2\to G$ and 
		$p_2:H_1\times H_2\to G$. The kernel of the $p_1'(e)-p_2'(e)$ is 
		the image under the diagonal map of $T_e(H_1)\cap T_e(H_2)$.
		So by the equality of the dimensions in Corollary \ref{dimension-of-kernel}
		we conclude $T_e(H_1\cap H_2)=T_e(H_1)\cap T_e(H_2)$.
	\end{proof}
	\begin{corollary}\label{subgroup-tangent}
		If $G$ is a definable strictly differentiable weak Lie group and 
		$H_1,H_2$ are subgroups, then there is $U\subset G$ an open neighborhood
		of $e$ such that $U\cap H_1=U\cap H_2$ if and only if $T_e(H_1)=T_e(H_2)$.
	\end{corollary}
	\begin{proof}
		By Corollary \ref{intersection-tangent} we get $T_e(H_3)=T_e(H_1)=T_e(H_2)$
		for $H_3=H_1\cap H_2$. Then as the inclusion $H_3\to H_1$ produces an 
		isomorphism of tangent spaces at the identity we conclude by the inverse
		function theorem \ref{inverse-function} that there is $U\subset G$ an open
		neighborhood of the identity, such that
		$U\cap H_1=U\cap H_3$. Note that this also uses that the topology of $H_1$
		and $H_3$ which makes them strictly differentiable definable 
		Lie groups coincides
		with the subgroup topology coming from $G$, see 
		Proposition \ref{subgroups-are-submanifolds}.
		
		Symmetrically we have $U'\cap H_2=U'\cap H_3$ for some open $U'$.
	\end{proof}
	
	Next we give the familiar 
	definition of the Lie bracket in $T_e(G)$ for the definable
	Lie group $G$, and show it forms a Lie algebra.
	\begin{definition}
		Suppose $G$ is a definable strictly differentiable weak Lie group.
		For $g\in G$ we consider the map $c_g:G\to G$ defined by $c_g(h)=ghg^{-1}$.
		Then $c_g$ is a definable group morphism and so it is strictly differentiable,
		see Proposition \ref{generic-regularity}.
		Its derivative produces a map $\Ad:G\to \mathrm{Aut}_K(T_e(G))$, $g\mapsto c_g'(e)$ which
		is a definable map and a group morphism by the chain rule and the equation
		$c_gc_h=c_{gh}$.
		Then $\Ad$ is strictly differentiable and so its derivative at $e$
		gives a linear map
		$\ad:T_e(G)\to \mathrm{End}_K(T_e(G))$.
		In other words this gives a bilinear map $(x,y)\mapsto \ad(x)(y)$, $T_e(G)\times T_e(G)\to T_e(G)$ denoted $(x,y)\mapsto [x,y]$.
		
		This map is called the Lie bracket.
	\end{definition}
	\begin{proposition}\label{local-lie-bracket}
		Let $G$ be a definable weak $T_2$-Lie group.
		Let $0\in U\subset K^n$ be an open set and $i:U\to G$ a $T_2$-diffeomorphism 
		of $U$ onto
		an open subset of $G$, that sends $0$ to $e$. Make $U$ into a local definable
		group via $i$.
		Then under the identification $i'(0):K^n\to T_e(G)$ we have that the Lie 
		bracket is characterized by the property
		$x\cdot y\cdot x^{-1}\cdot y^{-1}=[x,y]+O(x,y)^3$, for $x,y\in U$.
	\end{proposition}
	\begin{proof}
		We have that the function $f(x,y)=x\cdot y\cdot x^{-1}$
		satisfies $f(0,y)=y$ and $f(x,0)=0$, so its Taylor approximation of order
		2 is of the form $f(x,y)=y+axy+O(x,y)^3$. Indeed it is of the form
		$a_0+a_1x+a_2y+a_3x^2+a_4xy+a_5y^2+O(x,y)^3$ and plugging $x=0$ and using
		the uniqueness of the Taylor approximation we get $a_0=a_5=0$ and $a_2=1$,
		and a similar argument with $y=0$ gives $a_1=a_3=0$, so 
		$f(x,y)=y+axy+O(x,y)^3$ as claimed.
		
		From the definition of $\Ad(x)$ we get 
		$f(x,y)=\Ad(x)y+O_x(y^2)$ where $O_x$ means the coefficient may depend on $x$.
		
		Note that the definition of $\ad(x)$ gives $\Ad(x)(y)=y+[x,y]+O(x^2y)$.
		Indeed, we have $\Ad(x)=\Ad(0)+\Ad'(0)(x)+O(x^2)=I+\ad(x)+O(x^2)$, where
		$I$ is the identity matrix, and evaluating at $y$ we conclude
		$\Ad(x)y=y+[x,y]+O(x^2y)$.
		
		We conclude that $y+axy+O(x,y)^3=y+[x,y]+O_x(y^2)$.
		This implies $[x,y]=axy$. See Lemma \ref{xy}.
		Now from $x\cdot y\cdot x^{-1}=y+axy+O(y,x)^3$, and the formula
		$x\cdot y^{-1}=x-y+b_0x^2+b_1xy+b_2y^2+O(x,y)^3$ (see Fact \ref{m'}), 
		we get 
		$x\cdot y\cdot x^{-1}\cdot y^{-1}=(x\cdot y\cdot x^{-1})\cdot y^{-1}=axy+b_3y^2+O(x,y)^3$.
		On the other hand if $c(x,y)=x\cdot y\cdot x^{-1}\cdot y^{-1}$ then 
		$c(0,y)=0$ implies 
		that $b_3=0$, as required.
	\end{proof}
	\begin{proposition}
		Let $G$ be a definable strictly differentiable weak Lie group. Then 
		$(T_e(G),[,])$ is a Lie algebra.
	\end{proposition}
	\begin{proof}
		We have to prove $[x,x]=0$ and the Jacobi identity.
		We will use the characterization of Proposition \ref{local-lie-bracket}
		(we may assume $G$ is $T_2$ by Proposition \ref{definable-is-lie}).
		$[x,x]=0$ now 
		follows immediately.
		
		The idea of proof of the  Jacobi identity
		is to express $xyz$ as $f(x,y,z)zyx$ in two different ways using associativity,
		the first one permutes from left to right, the second permutes $yz$ and then
		permutes from left to right. The details follow.
		
		Writing $c(x,y)=xyx^{-1}y^{-1}$ one has
		\begin{center}
			$xyz=c(x,y)yxz=c(x,y)yc(x,z)zx=c(x,y)([y,c(x,z)]+O(y,c(x,z))^3)c(x,z)yzx=
			c(x,y)([y,[x,z]]+O(x,y,z)^4)c(x,z)c(y,z)zyx=
			(c(x,y)+[y,[x,z]]+c(x,z)+c(y,z)+O(x,y,z)^4)zyx$.
		\end{center}
		At the last step we use the formula $xy=x+y+O(x,y)^2$
		And on the other hand
		\begin{center}
			$xyz=xc(y,z)zy=([x,c(y,z)]+O(x,c(y,z))^3)c(y,z)xzy=
			([x,[y,z]]+O(x,y,z)^4)c(y,z)c(x,z)zxy=
			([x,[y,z]]+O(x,y,z)^4)c(y,z)c(x,z)zc(x,y)yx=
			([x,[y,z]]+O(x,y,z)^4)c(y,z)c(x,z)([z,[x,y]]+O(x,y,z)^4)c(x,y)zyx=
			([x,[y,z]]+c(y,z)+c(x,z)+c(x,y)+[z,[x,y]]+O(x,y,z)^4)zyx$.
		\end{center}
		From this we get
		$[y,[x,z]]=[x,[y,z]]+[z,[x,y]]+O(x,y,z)^4$ and from the uniqueness of Taylor 
		expansions we obtain 
		$[y,[x,z]]=[x,[y,z]]+[z,[x,y]]$ which is the Jacobi identity.
	\end{proof}
	Given a strictly differentiable definable weak Lie group $G$, we denote
	$Lie(G)$ the tangent space $T_e(G)$ considered as a Lie algebra with the Lie
	bracket $[x,y]$.
	
	\section{Definable fields}\label{fields-section}
	In this section we prove that if $L$ is a definable field in a 1-h-minimal valued field
	then, as a definable field, $L$ isomorphic to a finite field 
	extension of $K$. This is result generalizes
	\cite[Theorem 4.2]{BaysPet}, where this is proved for real closed valued fields,
	and  \cite[Theorem 4.1]{PilQp} where this is proven for $p$-adically closed
	fields. 
	
	With the terminology and results we have developed in the previous section the main ingredients of the proof are similar to those appearing in the classification of infinite fields definable in o-minimal fields, \cite[Theorem 1.1]{OtPePi}. 
	\begin{lemma}\label{definable-subfield}
		Suppose $K$ is 1-h-minimal, $L\sub  K$  a definable subfield. Then $L=K$.
	\end{lemma}
	\begin{proof}
		$L$ is a definable set which is infinite because the characteristic of $K$ is $0$.
		We conclude that there is a nonempty open ball $B\subset L$, for example
		by dimension theory, item 2 of Proposition \ref{dimension}.
		The field generated by a nonempty open ball is $K$.
		Indeed $B-B$ contains a ball around the origin $B'$, 
		$C=B'\setminus \{0\}^{-1}$ is the complement of 
		a closed ball, and $C-C=K$.
	\end{proof}
	\begin{lemma}\label{k-is-rigid}
		Suppose $K$ is 1-h-minimal.
		Let $F_1$ and $F_2$ be finite extensions of $K$, and consider them as definable fields
		in $K$. If $\phi:F_1\to F_2$ is a definable field morphism, then $\phi$ is a morphism
		of $K$ extensions, in other words it is the identity when restricted to $K$.
	\end{lemma}
	\begin{proof}
		The set $\{x\in K: \phi(x)=x\}$ is a definable subfield of $K$, so 
		Lemma \ref{definable-subfield} give the desired conclusion.
	\end{proof}
	\begin{proposition}\label{field}
		Suppose $K$ is 1-h-minimal and
		$F$ is a definable field. Then $F$ is isomorphic as a definable field to a finite 
		extension of $K$. The forgetful functor from finite $K$-extensions to definable fields
		is an equivalence of categories.
	\end{proposition}
	\begin{proof}
		That the functor is full is Lemma \ref{k-is-rigid}.
		
		Let $F$ be a definable field. By Proposition \ref{definable-is-lie} we have that
		$(F,+)$ is a definable strictly differentiable weak Lie group.
		If $a\in F$ the map $L_a:x\mapsto ax$ is a definable group morphism and so it is strictly
		differentiable, by the fullness in Proposition \ref{definable-is-lie}.
		We get a definable map $f:F\to M_n(K)$ defined as $a\mapsto L_a'(0)$.
		By the chain rule we have $f(ab)=f(a)f(b)$ for all $a,b\in F$. Clearly $f(1)=1$.
		Finally one has $f(a+b)=f(a)+f(b)$ (the derivative of multiplication 
		$G\times G\to G$
		in a Lie group is the sum map, see for instance Fact \ref{m'}).
		We conclude that $f$ is a ring map, and because $F$ is 
		a field it is injective. If we set $i:K\to M_n(K)$ given by $i(k)=kI$ where $I$
		is the identity matrix, then $i^{-1}f(F)\subset K$ is a definable subfield of $K$,
		and so by Lemma \ref{definable-subfield} one has $i(K)\subset f(F)$.
		So $F/K$ is a finite field extension as required.
	\end{proof}

	\section{One dimensional groups are finite by abelian by finite}
	\label{one-dimensional-section}
	In this section we prove that if $K$ is a 1-h-minimal valued field and $G$
	is a one dimensional group definable in $K$ then $G$ is finite-by-abelian-by-finite. This  generalizes \cite[Theorem 2.5]{PiYao} where it is proved
	that one dimensional groups definable in
	p-adically closed fields are abelian-by-finite. 
	This result is analogous to 
	\cite[Corollary 2.16]{Pi5} where it is shown that a one dimensional 
	group definable in an o-minimal structure is abelian-by-finite.
	
	The proof here is not a straightforward adaptation of either, 
	since we do not assume NIP, making the argument more involved.
	\begin{definition}\label{cw}
		Let $G$ be a group. We let $C^w$  denote the set of elements $x\in G$
		whose centralizer, $c_G(x)$, has finite index in $G$.
	\end{definition}
	
	Note that $C^w$ is a characteristic subgroup of $G$.
	
	\begin{lemma}\label{abelian-core}
		Suppose $G$ is an (abstract) group. 
		Take $C^w$ as in Definition \ref{cw}, 
		$Z$ its center.
		Then $C^w$ and $Z$ are characteristic groups of $G$, and $Z$ is commutative.
		Moreover $Z$ has finite index in $G$ if and only if
		$G$ is abelian-by-finite.
		
		When $G$ is definable in a geometric theory, $C^w$ and $Z$ are definable.
		Also $x\in C^w$ if and only if $\dim(c_G(x))=\dim(G)$.
	\end{lemma}
	\begin{proof}
		It is clear that $C^w$ and $Z$ are characteristic, and that $Z$ is abelian. So, in particular, if $[G:Z]<\infty$ then $G$ is abelian-by-finite. On the other hand, if $A$ is an abelian subgroup of finite then $A\subset C^w$,
		as $A\subset c_G(a)$ for every $a\in A$. 
		If $a_1,\dots, a_n$ are a set of representatives for left cosets of $A$
		in $C^w$ then $\bigcap_{k=1}^n c_G(a_k)\cap A\subset Z$, and as $a_k\in C^w$,
		the $c_G(a_k)$ have finite index in $G$, and so $Z$ has finite index in $G$.
		
		If $G$ is definable in a geometric theory note that $x\in C^w$ if and only if
		$x^G$, the orbit of $G$ under conjugation, is finite. This is because the fibers of the map
		$x\mapsto x^g$ are cosets of $c_G(x)$.
		So $C^w$ is definable because a geometric theory eliminates the exist 
		infinity quantifier. We also get that if $\dim(c_G(x))=\dim(G)$ 
		then $c_G(x)$ is of finite index.
	\end{proof}
	\begin{lemma}\label{locally-constant-pregeometry}
		Suppose $f:X\times Y\to Z$ is a function definable in a pregeometric theory.
		Denote $n=\dim(X)$ and $m=\dim(Y)$.
		Suppose for all $x\in X$ the nonempty fibers of the function 
		$f_x(y)=f(x,y)$ have
		dimension
		$m$.
		Suppose that for all $y\in Y$ the nonempty fibers of the function 
		$f_y(x)=f(x,y)$ have
		dimension $n$.
		Then $f$ has finite image.
	\end{lemma}
	\begin{proof}
		We claim that the nonempty fibers of $f$ have dimension $n+m$.
		Indeed, if $(x_0,y_0)\in X\times Y$, then $f^{-1}f(x_0,y_0)$
		contains $\cup_{x\in f_{y_0}^{-1}f_{y_0}(x_0)}\{x\}\times f_x^{-1}f_x(y_0)$,
		so we conclude by the additivity of dimension.
	\end{proof}
	\begin{lemma}\label{c-is-finite}
		Suppose $G$ is definable in a pregeometric theory and $G=C^w$.
		Then the image of the commutator map $c:G\times G\to G$ is finite.
	\end{lemma}
	\begin{proof}
		The commutator map $c(x,y)$ is constant when $x$ is fixed and $y$ varies
		over a right coset of $c_G(y)$, and it is constant when $y$ is fixed and 
		$x$ varies over a right coset of $c_G(x)$.
		This implies that the image of $c$ is finite,
		see Lemma \ref{locally-constant-pregeometry}.
	\end{proof}
	\begin{lemma}\label{almost-commutative}
		Suppose $G$ is an $n$-dimensional group definable 
		in a pregeometric theory such that  $G=C^w$. Then there is a definable 
		characteristic subgroup, $G_1$, of finite index
		with a characteristic finite subgroup $L$, central in $G_1$, such that
		$G_1/L$ is abelian.
		If $Z$ is the center of $G_1$ then $Z/L$ contains $(G_1/L)^m$, the $m$-th powers of $G_1/L$, for some $m$. 
		
		If the theory is NIP then the center of $G$ has finite index.
	\end{lemma}
	\begin{proof}
		If the theory has NIP then the center is finite index by Baldwin-Saxl
		(e.g., \cite[Lemma 1.3]{PoiGroups}).
		Indeed, $Z=\bigcap_{g\in G}c_G(g)$ is an intersection of a definable family subgroups, each of which is finite index by the assumption $G=C^w$.
		So, as $G$ is NIP one gets that $Z$ is the intersection of finitely many of the
		centralizers and $[G:Z]<\omega$.
		
		In general, 
		by  Lemma \ref{c-is-finite} we know that $c(G,G)$ is finite. The centralizer of
		$c(G,G)$ is $G_1$ and has finite index in $G$ by the hypothesis that $G=C^w$.
		Clearly $G_1$ is characteristic in $G$, so we may replace $G$ by $G_1$
		and assume that
		$c(G,G)$ is contained in the center of $G$.
		
		In this case we prove that $c(G,G)$ generates a 
		finite central characteristic group $L=D(G)$.
		Indeed, since $c(G,G)$ is central, a simple computation shows
		$c(gh,x)=c(g,x)c(h,x)$ for all $g,h,x\in G$. 
		It follows that 
		$c(g,h)^m=c(g^m,h)$ is in $c(G,G)$. Thus $c(g,h)$ has finite order.
		As $c(G,G)$ is central with elements of finite order, 
		the group it generates is central and finite. It is obviously characteristic.
		
		We also see that if $m$ is the order of $D(G)$ then $g^m\in Z$ for all
		$g\in G$. This is because 
		$c(g^m,h)=c(g,h)^m=1$, so $(G/D(G))^m$ is contained
		in $Z/D(G)$ as required.
	\end{proof}
	\begin{lemma}\label{abelian-n-torsion}
		Suppose $G$ is an $n$-dimensional 
		abelian group definable in a 1-h-minimal theory.
		Then the $m$-torsion of $G$ is finite and $G^m\subset G$ is a subgroup
		of dimension $n$.
	\end{lemma}
	\begin{proof}
		The map  $x\mapsto x^m$ is a definable group morphism with invertible
		derivative at the identity, see for instance Fact \ref{m'},
		so by Corollary \ref{dimension-of-kernel} we
		get that the $m$-torsion of $G$ is finite, and so by additivity of dimension
		$\dim(G^m)=\dim(G)=n$.
	\end{proof}
	\begin{lemma}\label{almost-center}
		Suppose $G$ is an $n$-dimensional
		group definable in a 1-h-minimal field.
		Then $C^w$ is the kernel of the map $\Ad:G\to GL_n(K)$.
		
		If the Lie algebra $\mathrm{Lie}(G)$ is abelian, 
		then $C^w$ has finite index.
	\end{lemma}
	\begin{proof}
		The first statement follows from Corollary \ref{map-germ-at-1} and 
		Lemma \ref{abelian-core}.
		
		If the Lie bracket is abelian, then, by the definition of the 
		Lie bracket, the derivative of $\Ad$ at $e$ is $0$.
		This means that $C^w=\ker(\Ad)$ 
		contains an open neighborhood of $e$ by Corollary 
		\ref{map-germ-at-1}, so $C^w$ is $n$-dimensional.
		As the kernel of $\Ad$ is $C^w$ we conclude that $C^w$ has finite index in 
		$G$ by the additivity of dimension.
	\end{proof}
	\begin{lemma}\label{center-finite-by-abelian}
		Suppose $G$ is an $n$-dimensional 
		finite-by-abelian group definable in a 1-h-minimal 
		theory. Then $G=C^w$ and 
		the center of $G$ has dimension $n$.
	\end{lemma}
	\begin{proof}
		Let $H$ be finite normal such that $G/H$ is abelian.
		By finite elimination of imaginaries in fields we have that 
		$G/H$ is definable.
		Also by Corollary \ref{dimension-of-kernel} we see that that the quotient map $p:G\to G/H$ induces
		an isomorphism of tangent spaces at the identity, and under this isomorphism
		$1=\Ad(p(g))=\Ad(g)$ for all $g\in G$. We conclude that $G=C^w$, by Lemma 
		\ref{almost-center}. Now if we apply Lemma \ref{almost-commutative}
		we get characteristic groups $L\subset G_1\subset G$ such that 
		$G/G_1$ 
		is finite, $L$ is finite and $G_1/L$ is abelian, and 
		$Z(G_1)/L\supset (G_1/L)^m$ for some $m$.
		Note  that $Z(G_1)$ contains a finite index subgroup of
		$Z(G)$, so we just have to see that $(G_1/L)^m$ has dimension $n$.
		
		This follows from Lemma \ref{abelian-n-torsion}.
	\end{proof}
	\begin{proposition}\label{abelian}
		Suppose $K$ is 1-h-minimal.
		Suppose $G$ is a strictly differentiable definable weak Lie group.
		Then $\mathrm{Lie}(G)$ is abelian 
		if and only if $G$ is finite-by-abelian-by-finite.
		In this case $G$ has characteristic definable subgroups
		$L\subset G_1\subset G$ such that $G/G_1$ is finite, $G_1/L$ is abelian,
		and $L$ is finite and central in $G_1$. Also if $Z$ is the center of $G'$,
		then $Z$ is $n$-dimensional, and $Z/L$ contains $(G_1/L)^m$ for some $m$.
		
		If $K$ is NIP then we may take $L=1$.
	\end{proposition}
	\begin{proof}
		This follows by putting the previous results together, 
		Lemmas \ref{almost-center}, \ref{almost-commutative}, 
		\ref{center-finite-by-abelian}.
	\end{proof}
	\begin{corollary}\label{one-dimensional}
		Suppose $G$ is a one dimensional group definable in a 1-h-minimal valued
		field. Then $G$ is finite-by-abelian-by-finite.
		If the theory is NIP then $G$ is abelian-by-finite.
	\end{corollary}
	\begin{proof}
		By Proposition \ref{definable-is-lie} we get that 
		$G$ is a strictly differentiable definable weak Lie group.
		The result now follows from Proposition \ref{abelian} because the only
		one dimensional Lie algebra is abelian.
	\end{proof}
	In the NIP case this corollary follows more directly from the fact 
	that a definable group is definably weakly Lie. Indeed, this implies that there
	is an element $x\in G$ with $x^n\neq e$ (because the derivative of the map
	$x\mapsto x^n$ at $e$ is $v\mapsto nv$, which is not equal to $0$, see
	Fact \ref{m'}). 
	By $\aleph_0$-saturation there is an $x\in G$ such that the group generated
	by $x$ is infinite. Then by
	\cite{PiYao}, Remark 2.4, one has that the centralizer of the centralizer of 
	$x$ has finite index and is abelian. Indeed, note that if
	$a\in c_G(x)$, $c_G(a)$ contains
	the group generated by $x$, and so it is of dimension $1$. 
	
	The last corollary shows that the classification of one dimensional abelian groups definable in 
	ACVF carried out in \cite{AcostaACVF} extends to all definable $1$-dimensional groups,  for $K$ of characteristic $0$ (see also the main result of \cite{HaHaPeGps}). I.e., since ACVF$_0$ is $1$-h-minimal and NIP, $1$-dimensional definable groups are abelian-by-finite, and the classification of definable $1$-dimensional abelian groups of \cite{AcostaACVF} applies.  
	We do not know if this corollary is true in ACVF$_{p,p}$.
	Similarly, the commutativity assumption is  unnecessary in the classification of $1$-dimensional groups definable in pseudo-local fields of
	residue characteristic $0$. As those are pure henselian they, too,  are $1$-h-minimal, so we may apply Proposition \ref{abelian}. To get the full result we observe that though pseudo-local fields are not NIP,  an inspection of the 
	list of the definable $1$-dimensional abelian groups, $A$, obtained
	in \cite{AcostaACVF} shows they are almost divisible (i.e.,  
	$nA$ has finite index in $A$ for all $n$). Therefore, in the notation of   
	Proposition \ref{abelian} the 
	center of $G_1$ has finite index in $G_1$, and so every one dimensional group
	is abelian-by-finite. 
 %\textcolor{red}{I did not understand how you concluded this from the classification of $1$-dimensional abelian groups. It seems that what you are claiming is not that commutativity need not be assumed, but that we get "abelian by finite" -- so a stronger conclusion.}
	
	\begin{question}
		If $G$ is finite-by-abelian, does the center of $G$ have finite index in $G$?
		This is true if the theory is NIP or if $nA$ has finite index in $A$
		for every abelian definable group, 
		by Lemmas \ref{center-finite-by-abelian} and \ref{almost-commutative}.
	\end{question}
	\begin{remark}
		ACVF$_{p,p}$ does not fit into the framework of 1-h-minimality.
		However, many of the ingredients in previous sections translate to this setting. For example: $K$ is geometric, a subset of $K^n$ has dimension $n$ if
		and only if it contains a nonempty open set, one-to-finite functions defined in an open set
		are generically continuous, functions definable in an open set are generically
		continuous, and $K$ is definably spherically complete.
		
		That one-to-finite functions are generically continuous follows from 
		the fact that $\acl(a)$ coincides with the field-theoretic algebraic closure of $a$
		and by a suitable result about continuity of roots. 
		That functions are generically continuous follows
		from the fact that $\dcl(a)$ is the Henselization of the perfect closure of $a$,
		so a definable function is definably 
		piecewise a composition of rational functions, inverse of the 
		Frobenius automorphism and roots of Hensel polynomials,
		all of these functions being continuous.
		
		However,
		the inverse of the Frobenius is not differentiable anywhere, so Proposition
		\ref{generic-regularity-0} does not hold. Also the Frobenius
		is an homeomorphism with $0$ derivative, so for example 
		Proposition \ref{locally-constant} does not hold.
	\end{remark}
	\bibliographystyle{plain}
	\bibliography{harvard}

\begin{thebibliography}{10}

\bibitem{AcostaACVF}
Juan~Pablo {Acosta}.
\newblock {One dimensional commutative groups definable in algebraically closed
  valued fields and in the pseudo-local fields}.
\newblock {\em arXiv e-prints}, page arXiv:2112.00430, December 2021.

\bibitem{BaysPet}
Martin Bays and Ya'acov Peterzil.
\newblock Definability in the group of infinitesimals of a compact {L}ie group.
\newblock {\em Confluentes Math.}, 11(2):3--23, 2019.

\bibitem{BWHal}
David {Bradley-Williams} and Immanuel {Halupczok}.
\newblock {Spherically complete models of Hensel minimal valued fields}.
\newblock {\em arXiv e-prints}, page arXiv:2109.08619, September 2021.

\bibitem{hensel-min}
Raf Cluckers, Immanuel Halupczok, and Silvain Rideau-Kikuchi.
\newblock Hensel minimality {I}.
\newblock {\em Forum Math. Pi}, 10:Paper No. e11, 2022.

\bibitem{hensel-minII}
Raf {Cluckers}, Immanuel {Halupczok}, Silvain {Rideau-Kikuchi}, and Floris
  {Vermeulen}.
\newblock {Hensel minimality II: Mixed characteristic and a diophantine
  application}.
\newblock {\em arXiv e-prints}, page arXiv:2104.09475, April 2021.

\bibitem{HaHaPeVF}
Yatir Halevi, Assaf Hasson, and Ya'acov Peterzil.
\newblock Interpretable fields in various valued fields.
\newblock {\em Adv. Math.}, 404:Paper No. 108408, 2022.

\bibitem{HaHaPeGps}
Yatir {Halevi}, Assaf {Hasson}, and Ya'acov {Peterzil}.
\newblock {On groups interpretable in various valued fields}.
\newblock {\em arXiv e-prints}, page arXiv:2206.05677, June 2022.

\bibitem{HruPil}
Ehud Hrushovski and Anand Pillay.
\newblock Groups definable in local fields and pseudo-finite fields.
\newblock {\em Israel J. Math.}, 85(1-3):203--262, 1994.

\bibitem{HrRid}
Ehud Hrushovski and Silvain Rideau-Kikuchi.
\newblock Valued fields, metastable groups.
\newblock {\em Selecta Math. (N.S.)}, 25(3):Paper No. 47, 58, 2019.

\bibitem{OtPePi}
Margarita Otero, Ya'acov Peterzil, and Anand Pillay.
\newblock On groups and rings definable in o-minimal expansions of real closed
  fields.
\newblock {\em Bull. London Math. Soc.}, 28(1):7--14, 1996.

\bibitem{Pi5}
Anand Pillay.
\newblock {On groups and fields definable in {$o$}-minimal structures}.
\newblock {\em J. Pure Appl. Algebra}, 53(3):239--255, 1988.

\bibitem{PilQp}
Anand Pillay.
\newblock On fields definable in {${\bf Q}_p$}.
\newblock {\em Arch. Math. Logic}, 29(1):1--7, 1989.

\bibitem{PiYao}
Anand Pillay and Ningyuan Yao.
\newblock A note on groups definable in the {$p$}-adic field.
\newblock {\em Arch. Math. Logic}, 58(7-8):1029--1034, 2019.

\bibitem{PoiGroups}
Bruno Poizat.
\newblock {\em Stable groups}, volume~87 of {\em Mathematical Surveys and
  Monographs}.
\newblock American Mathematical Society, Providence, RI, 2001.
\newblock Translated from the 1987 French original by Moses Gabriel Klein.

\bibitem{SimWal}
Pierre Simon and Erik Walsberg.
\newblock Tame topology over dp-minimal structures.
\newblock {\em Notre Dame J. Form. Log.}, 60(1):61--76, 2019.

\bibitem{WilComp}
A.~J. Wilkie.
\newblock A theorem of the complement and some new o-minimal structures.
\newblock {\em Selecta Math. (N.S.)}, 5(4):397--421, 1999.

\end{thebibliography}
\end{document}